\documentclass[12pt]{amsart}
\usepackage{a4wide,url}
\usepackage{amsfonts,times}
\usepackage{amssymb}
\usepackage{amsmath}
\usepackage{amsthm}
\usepackage{subcaption}
\usepackage{graphicx}
\usepackage{xcolor}
\usepackage{verbatim}
\usepackage{gensymb}
\usepackage{hyperref}

\definecolor{citecol}{HTML}{802D1B}
\definecolor{tableofcontent}{HTML}{1F4A83}
\definecolor{urlcol}{HTML}{2470D8}

\newtheorem{theorem}{Theorem}[section]
\newtheorem{lemma}[theorem]{Lemma}
\theoremstyle{definition}
\newtheorem{proposition}[theorem]{Proposition}
\newtheorem{corollary}[theorem]{Corollary}
\theoremstyle{remark}
\newtheorem{remark}{Remark}

\numberwithin{equation}{section}

\newcommand{\R}{\mathbb{R}}
\newcommand{\Sph}{\mathbb{S}^2}
\newcommand{\calR}{\mathcal{R}}
\DeclareMathOperator{\tr}{tr}
\def\dd{{\mathrm{d}}}
\newcommand{\inner}[1]{{\left\langle #1 \right\rangle}}
\newcommand{\vKmaxabs}{v_{{\rm max}}}
\newcommand{\vK}{v}
\newcommand{\avJ}{{a}^v_J} 
\newcommand{\vKmin}{v^{{\rm min}}}
\newcommand{\vKmax}{v^{{\rm max}}}

\newcommand{\amin}{\widehat{\ba}}
\newcommand{\Qmin}{\widehat{Q}}
\newcommand{\bzero}{\boldsymbol{0}}
\newcommand{\ba}{{\boldsymbol{a}}}
\newcommand{\bb}{{\boldsymbol{b}}}
\newcommand{\bc}{{\boldsymbol{c}}}

\newcommand{\br}{{\boldsymbol{r}}}
\newcommand{\bq}{{\boldsymbol{q}}}
\newcommand{\bu}{{\boldsymbol{u}}}
\newcommand{\by}{{\boldsymbol{y}}}

\newcommand{\beps}{{\boldsymbol{\varepsilon}}}
\newcommand{\balpha}{{\boldsymbol{\alpha}}}

\newcommand{\aCov}{\Omega} 
\newcommand{\eCov}{\Upsilon} 

\title[Removing the mask]{Removing the mask -- reconstructing a scalar field on the sphere from a masked field}
\author[Hamann]{Jan Hamann}
\address{School of Physics, UNSW Sydney, Australia.}
\email{jan.hamann@unsw.edu.au}
\author[Le Gia, Sloan, Womersley]
{Quoc T. Le Gia \and Ian H. Sloan \and Robert S. Womersley}
\address{School of Mathematics and Statistics, UNSW Sydney, Australia.}
\email{qlegia@unsw.edu.au, i.sloan@unsw.edu.au, r.womersley@unsw.edu.au}
\date{5 Sep 2024}

\begin{document}

\maketitle

\begin{abstract}
The paper analyses a spectral approach to reconstructing 
a scalar field on the sphere, given only information about a masked version of
the field together with precise information about the (smooth) mask. The theory is developed for a general mask, and later specialised to the case of an axially symmetric mask. Numerical experiments are given for the case of an axial mask motivated by the cosmic microwave background, assuming that the underlying field is a realisation of a Gaussian random field with an artificial angular power spectrum of moderate degree ($\ell \le 100$).  The recovery is highly satisfactory in the absence of noise and even in the presence 
of moderate noise. 
\end{abstract}

\section{Introduction}
In this paper, we study the reconstruction of a scalar field (for example temperature or pressure) on
the unit sphere, given (possibly noisy) data on  a masked version of the
field, together with precise knowledge of the mask.  The underlying motivation is the cosmic microwave background (CMB) for which
the temperature observations, so important for the modern understanding of
the early universe, are obscured over substantial portions of the sky by
our own Milky Way, creating the need for masking some portions before attempting reconstruction.

The paper aims at a proof-of-concept for a new spectral approach to such problems.  While the theory is general, the numerical experiments are restricted to the case of an axially symmetric mask, and limit the field's angular power spectrum to polynomial degree $\ell \le  100$,
corresponding to an angular resolution of approximately $2^{\degree}$.
Within these limitations, the recovery is shown to be highly satisfactory in the no-noise case, and also in the case of moderate noise.

There is a rich literature on the inpainting problem for the particular case of the CMB~\cite{Abrial:2008mz,Starck:2012gd}, with techniques based on harmonic methods~\cite{Bielewicz:2004en,Inoue:2008qf}, iterative methods~\cite{Nishizawa:2013uwa,Gruetjen:2015sta}, constrained Gaussian realisations~\cite{Kim:2012iq,Bucher:2011nf},
group sparse optimisation methods~\cite{LiChen2023} or
neural networks~\cite{Yi_2020,Puglisi:2020deh,Sadr:2020rje,Montefalcone:2020fkb,Wang:2022ybb}.

Closest to the present approach is the work of Alonso et al.~\cite{Alonso:2018jzx}, which however differs in aim (which was to recover the angular power spectrum from a knowledge of the ``pseudo $C_\ell$'', through pooling together Fourier coefficients of different degrees).

The problem is formulated in the next section, by reducing the problem to that of solving a large ill-posed linear system. 
Section~\ref{sec:PropE} establishes properties of the matrix in that linear system. 
Section~\ref{sec:Solving} outlines our stochastic approach to the solution of the linear system. Because the full linear system is currently beyond our resources, in Section~\ref{sec:axsym} we obtain a large reduction in difficulty by specialising to the case of an axially symmetric mask. The final section is devoted to numerical experiments in both the no-noise and noisy cases.

\section{The problem setting}\label{sec:Setting}

Taking $\br$ to be any point in the unit sphere $\mathbb{S}^2$ in $\R^3$
(i.e. $\br$ is a unit vector in $\R^3$),
it can be represented in spherical coordinates as
\[
\br = (\sin\theta \cos\phi, 
\sin\theta \sin \phi, \cos\theta),
\quad \theta \in [0,\pi],\; \phi \in [0,2\pi).
\]

The underlying real scalar field $a(\br)$ is
assumed to be partially obscured  by a known mask $v = v(\br)$ to give a masked field $a^v = a^v(\br) := a(\br) v(\br)$. 

It is well-known \cite{Marinucci_Peccati2011} that the space of square-integrable functions $L_2(\mathbb{S}^2)$ admits 
an orthonormal basis formed by spherical harmonics
$$\{Y_{\ell,m}: \ell=0,\ldots; m=-\ell,\ldots,\ell\},$$ where the $Y_{\ell,m}$ is given explicitly by
\begin{equation}
\begin{aligned}
Y_{\ell,m}(\theta,\phi) &= 
\sqrt{ \frac{2\ell+1}{4\pi}
\frac{(\ell-m)!}{(\ell+m)!}} P^{m}_{\ell} (\cos\theta) e^{im \phi}, \quad m \ge 0,\\
Y_{\ell,m}(\theta,\phi) &= (-1)^m \overline{Y}_{\ell,-m}
(\theta,\phi), \quad m < 0,
\end{aligned}
\end{equation}
where $\{P^m_{\ell}\}$ denote the associated 
Legendre function, which is defined in terms of Legendre polynomials $\{P_{\ell}:\ell-0,1,\ldots\}$ by
\[
P^m_{\ell}(t) = 
(-1)^m (1-t^2)^{m/2} \frac{\dd^m}{\dd t^m} P_{\ell}(t),
\quad m=0,1,\ldots,\ell; \;\ell=0,1,2,\ldots
\]

We assume that the incompletely known field $a$ is a spherical polynomial of degree at most $L$, and hence expressible as an 
expansion in terms of orthonormal (complex) spherical harmonics 
$Y_{\ell,m}(\br)$ of degree $\ell \le L$,  
\begin{equation}\label{eq:a_and_v}
a(\br) = \sum_{\ell=0}^L \sum_{m = -\ell}^\ell a_{\ell,m} Y_{\ell,m}(\br),
\qquad \br\in\Sph,
\end{equation}
where 
\[
a_{\ell,m} := \int_{\Sph} a(\br)\overline{Y_{\ell,m}(\br)} 
\dd S(\br) ,
\]
where $\dd S(\br) = \sin\theta \dd\theta \dd\phi$, following from the orthonormality of the spherical harmonics,
\[
\int_ {\Sph}Y_{\ell,m}(\br) \overline{Y_{\ell',m'}(\br)} \dd S(\br) =
\delta_{\ell, \ell'} \delta_{m,m'},
\quad \ell, \ell' \ge 0,\; m = -\ell,\ldots,\ell,\;
 m' =  -\ell',\ldots,\ell'.
\]

We assume that the mask 
can be adequately approximated by a partial sum of spherical harmonics series
as\footnote{The mask $v$ is
typically identically $1$ or $0$ over some parts of the sphere,
and hence not polynomial, but expressible as a uniformly convergent infinite sum of the same form as in \eqref{eq:v}.} 
\begin{equation}\label{eq:v}
\vK(\br) = \sum_{k=0}^K \sum_{\nu = -k}^k v_{k,\nu} Y_{k,\nu}(\br),
\end{equation}
where 
\[
 v_{k,\nu} := \int_{\Sph} v(\br)\overline{Y_{k,\nu}(\br)} \dd S(\br) .
\]
Thus the masked signal $a^v(\br) = a(\br) v(\br)$
is expressible as a spherical harmonic expansion of degree $L + K$,
\begin{equation}\label{eq:avfun}
    a^v(\br) = a(\br) v(\br) = \sum_{j=0}^{L+K}\sum_{\mu=-j}^{j}a_{j,\mu}^{v} Y_{j,\mu}(\br), \qquad \br\in\Sph,
\end{equation}
where
\begin{align}\label{eq:av}
    a_{j,\mu}^v
    &:= \int_{\Sph} \avJ(\br)\overline{Y_{j,\mu}(\br)} \dd S(\br)
    = \int_{\Sph} a(\br)v(\br)\overline{Y_{j,\mu}(\br)} \dd S(\br)\\
    & = \int_{\Sph}
    \left(\sum_{\ell=0}^{L}
    \sum_{m=-\ell}^{\ell}
    a_{\ell,m}Y_{\ell,m}(\br)\right)
    \left(\sum_{k=0}^{K}\sum_{\nu=-k}^{k}
    v_{k,\nu}Y_{k,\nu}(\br)\right)
    \overline{Y_{j,\mu}(\br)}\dd S(\br)\nonumber\\
    &= \sum_{\ell=0}^{L}
    \sum_{m=-\ell}^{\ell}
    \sum_{k=0}^{K}
    \sum_{\nu=-k}^{k}a_{\ell,m}v_{k,\nu}
    \int_{\Sph} Y_{\ell,m}(\br) Y_{k,\nu}(\br) \overline{Y_{j,\mu}(\br)}\dd S(\br)\nonumber\\
    &= \sum_{\ell=0}^{L}
    \sum_{m=-\ell}^{\ell} E_{j,\mu; \ell,m} a_{\ell,m},
    \nonumber
\end{align}
where, for $j = 0, \ldots, L+K$, $\mu = -j, \ldots, j$
and $\ell = 0,\ldots,L$, $m = -\ell,\ldots,\ell$,
\begin{equation}\label{eq:Edefined}
E_{j,\mu; \ell,m} = \sum_{k=0}^{K}
    \sum_{\nu=-k}^{k}
    \int_{\Sph} Y_{\ell,m}(\br) Y_{k,\nu}(\br) \overline{Y_{j,\mu}(\br)}\dd S(\br)  \; v_{k,\nu}. 
\end{equation}

Thus the essential task in the reconstruction is to solve as accurately as possible the large linear system 
\begin{equation}\label{eq:aequalsEa}
\sum_{\ell=0}^{L}
    \sum_{m=-\ell}^{\ell}
    E_{j,\mu;\ell,m}a_{\ell,m} = a_{j,\mu}^v,
\end{equation}
where $E$ is defined by \eqref{eq:Edefined}.
 
Equation \eqref{eq:aequalsEa} can be
considered as an overdetermined (but possibly not full-rank) set of linear
equations for the $a_{\ell, m}$. We will find it convenient to replace the
upper limit $L+K$ in
\eqref{eq:aequalsEa} 
by a more flexible upper limit
$J$ with $L \leq J \leq L+K$.
Then the equation can be written as
\begin{equation}\label{eq:aeqav_math}
E \ba = \ba^v,
\end{equation}
where $E$ is a $(J+1)^2 \times (L+1)^2$ matrix, 
\begin{equation}\label{eqn-ba}
\ba = (a_{\ell,m}, \ell = 0,\ldots,L,
  m = -\ell,\ldots,\ell) \in \mathbb{C}^{(L+1)^2}.
\end{equation}
and
\begin{equation}\label{eqn-bav}
\ba^v = (a^v_{j,\mu}, \ell = 0,\ldots,J,
  m = -\ell,\ldots,\ell) \in \mathbb{C}^{(J+1)^2}.
\end{equation}

In Section~\ref{sec:Solving} we come to the most challenging part of the paper, which is the approximate solution of the ill-posed linear system, and before it the computation of the matrix $E$.  Before then, however, it is useful to establish properties of the matrix $E$.

\section{Properties of $E$} \label{sec:PropE}
This section summarizes useful properties of the matrix $E$ 
defined in \eqref{eq:Edefined}.
We first note that the product of three spherical harmonics  (known as a Gaunt coefficient) can be evaluated in terms of Wigner 3j symbols, see
eg. ~\cite[Eq.~34.4.22]{NIST:DLMF} or \cite{Alonso:2018jzx}.  Explicitly,
\begin{align}\label{eq:D}
D_{ \ell,m; k, \nu; j,\mu}
&:= \int_{\Sph} Y_{\ell,m}(\br) Y_{k,\nu}(\br) \overline{Y_{j,\mu}(\br)}\dd S(\br)\\
&= (-1)^\mu \sqrt{\frac{(2\ell+1)(2k+1)(2j+1)}{4\pi}}\nonumber\\
    & \hspace{.3cm}\times \begin{pmatrix}
    \ell & k & j\\
    0 & 0 & 0
    \end{pmatrix}
    \begin{pmatrix}
    \ell & k & j\\
    m & \nu & -\mu
    \end{pmatrix} .\nonumber
\end{align}
As important special cases, 
\begin{equation}\label{eq:Dprops}
D_{ \ell,m; k, \nu; j,\mu} = 0 \, \mbox{ if}\,
\begin{cases}
j+\ell+k \quad \mbox{is odd, or}\\
k < |j - \ell|,\mbox{ or}\\
k > j + \ell,\mbox{ or}\\
m  + \nu \ne  \mu.
\end{cases}
\end{equation}

 The following lemma gives  several elementary properties of the matrix $E$, beginning with an explicit integral  expression in terms of the mask function $v$.
\begin{lemma}\label{lem:Eprop}
The elements of the matrix $E$ satisfy
\begin{equation}\label{eq:Eint}
E_{j,\mu;\ell,m} =
  \int_{\Sph} \overline{Y_{j, \mu}(\br)}
  Y_{\ell,m }(\br) v(\br)\, \dd S(\br),
\end{equation}
and, for a real mask $v$,
\begin{eqnarray}
E_{\ell,m; j, \mu} & = &
\overline{E_{j, \mu; \ell, m}} \, ,  \label{eq:Eherm} \\
E_{j,\mu;\ell,m} & = &
 \sum_{k=0}^{K} \left[
   D_{ \ell,m; k, 0; j,\mu} v_{k,0} + 
   2 \sum_{\nu=1}^{k} \Re(D_{ \ell,m; k, \nu; j,\mu} v_{k,\nu}) \right], \label{eq:Ereal} \\
E_{j, \mu; \ell,-m}  & = & 
(-1)^{m-\mu} \overline{E_{j,-\mu;\ell,m}} \, . \label{eq:Enegm}
\end{eqnarray}
\end{lemma}
\begin{proof}
Firstly, from the definition of $E$ in \eqref{eq:Edefined} we have
\begin{align}
    E_{j,\mu;\ell,m} &=
    \sum_{k=0}^{K} \sum_{\nu = -k}^k \int_{\Sph}
    Y_{\ell, m}(\br) Y_{k, \nu}(\br)
    \overline{Y_{j, \mu}(\br)}
    \, \dd S(\br)\: v_{k, \nu}\nonumber\\
    &= \int_{\Sph}
    \overline{Y_{j, \mu}(\br)}
    Y_{\ell, m}(\br) \sum_{k=0}^{K}\sum_{\nu = -k}^k
    v_{k, \nu} Y_{k, \nu}(\br) \,\dd S(\br)\nonumber\\
    &= \int_{\Sph}
    \overline{Y_{j, \mu}(\br)} Y_{\ell,m }(\br)
    \vK(\br) \, \dd S(\br),\nonumber
\end{align}
establishing \eqref{eq:Eint}.
From the  definition of $D_{\ell,m;k,\nu;j,\mu}$ in \eqref{eq:D} as an integral, together with the spherical harmonic property 
\begin{equation}\label{eqn-Yell-m}
Y_{\ell,-m}(\br) = (-1)^m \overline{Y_{\ell,m}(\br)},    
\end{equation}
it follows that
\[
D_{j,\mu; k,\nu;\ell,m;} = (-1)^\nu \overline{D_{\ell,m;k,-\nu;j,\mu}}.
\]
Because  both the mask $v$ and the field $a$ are real we have
\begin{equation}\label{eq:a_symm}
a_{\ell,m} = (-1)^m \overline{a_{\ell,-m}},\quad
v_{k,\mu} = (-1)^\mu \overline{v_{k,-\mu}},
\end{equation}
for all relevant values of $\ell, m, k$ and $\mu$,
and \eqref{eq:Eherm} then follows from \eqref{eq:Edefined}.
Also, from \eqref{eq:D} and 
\eqref{eqn-Yell-m},
\[
D_{ \ell,-m; k, \nu; j,\mu}  =
(-1)^m \overline{D_{ \ell,m; k, \nu; j,\mu}}, \quad
D_{ \ell,m; k, -\nu; j,\mu}  =
(-1)^\nu \overline{D_{ \ell,m; k, \nu; j,\mu}},
\]
so \eqref{eq:a_symm} gives
\[
  D_{ \ell,m; k, -\nu; j,\mu}v_{k,-\nu} =
\overline{D_{ \ell,m; k, \nu; j,\mu} v_{k,\nu}},
\]
and \eqref{eq:Edefined} then yields \eqref{eq:Ereal}.
Finally, for a real mask \eqref{eq:Enegm} follows easily from \eqref{eq:Eint}.
 
\end{proof}

\subsection{Singular values of $E$}
This subsection gives upper bounds on the singular values of the rectangular matrix $E$ in terms of the real mask $v$.
We use $E^*$ to denote the complex conjugate transpose of the matrix $E$.

\begin{theorem}\label{thm:svbnd}
Assume that the mask $v(\br)$ is approximated by the
partial sum of its spherical Fourier series of degree
$K \ge 1$
and that $L \leq J\le L+K$ is the degree of the approximation
$\avJ(\br)$ to the masked field $a^v(\br)$,
so that $E$ is a $(J+1)^2$ by $(L+1)^2$ matrix.
The singular values $\sigma$ of $E$ satisfy
\begin{equation}\label{eq:svbnd}
    0 \leq \sigma \leq \vKmaxabs,
\end{equation}
where
\[
\vKmaxabs := \max_{\br\in \Sph}\; |v(\br)| .
\]
\end{theorem}
\begin{proof}
Let $\bu\neq \bzero$ be an eigenvector of
the positive semi-definite Hermitian matrix $E^*E$ corresponding to the non-negative real eigenvalue $\sigma^2$,  so $E^* E \bu = \sigma^2 \bu$.
Then, using \eqref{eq:Eint}, and writing the elements of
$\bu$ as $u_{\ell,m}, \ell = 0,\ldots,L,\; m = -\ell,\ldots, \ell$, we have
\[
  (E \bu)_{j, \mu} = 
  \sum_{\ell=0}^L \sum_{m=-\ell}^\ell u_{\ell,m}
  \int_{\Sph} \overline{Y_{j,\mu}(\br)}
  Y_{\ell,m}(\br)v(\br) \dd S(\br)
   = \int_{\Sph} \overline{Y_{j,\mu}(\br)}
  u(\br) v(\br) \dd S(\br), 
\]
for $j = 0,\ldots,J$, $\mu = -j,\ldots,j$, where 
\[
\  u(\br) :=  \sum_{\ell=0}^L \sum_{m=-\ell}^\ell
   u_{\ell,m} Y_{\ell,m} (\br),
\]
giving
\[
  \bu^* E^* E \bu = \|E\bu\|_{\ell_2}^2  =
  \sum_{j = 0}^J \sum_{\mu=-j}^j
  \left| \int_{\Sph} u(\br) v(\br) \overline{Y_{j,\mu}(\br)} \dd S(\br) \right|^2.
\]
Thus, using Parseval's identity for $u\, v$ and then $u$,
\begin{align}
\label{eq:svsq}
  \sigma^2 \| \bu \|_{\ell_2}^2 =
  \bu^* (\sigma^2 \bu) = \bu^* (E^* E \bu) &  =
\sum_{j = 0}^J \sum_{\mu=-j}^j
  \left| \int_{\Sph} u(\br) v(\br) \overline{Y_{j,\mu}(\br)} \dd S(\br) \right|^2 \\
  & = \| u \; v \|^2_{L^2} \nonumber \\
  & \leq  (\vKmaxabs)^2 \| u \|_{L^2}^2
  = (\vKmaxabs)^2 \| \bu \|_{\ell^2}^2 . \nonumber
\end{align}
This gives the upper bound \eqref{eq:svbnd}
on the singular values $\sigma$ of $E$.

\end{proof}

\subsection{Eigenvalues of $E$}
In this subsection we take $J = L$, making the matrix $E$ square. From the second statement in Lemma~\ref{lem:Eprop}, $E$ is Hermitian thus its eigenvalues are real, and
eigenvectors belonging to distinct eigenvalues are orthogonal.

Let $\lambda\in \R$ be an eigenvalue of $E$ and let $\bq\neq \bzero$ be a
corresponding eigenvector, thus
\begin{equation*}
    E \bq = \lambda\bq.
\end{equation*}
The following result provides both lower and upper bounds on $\lambda$,
in terms of the minimum and maximum values of the mask $v$.

\begin{theorem}\label{thm:evbnd}
Assume that the mask $v$ is approximated by the
partial sum of its spherical Fourier series of degree $K\geq 1$.
Assume also that $J = L$, so that the matrix $E$ is square. Then the
eigenvalues of $E$ lie in the interval $(\vKmin, \vKmax)$, where
\[
\vKmin := \min_{\br\in \mathbb{S}^2}\;v(\br),
\qquad
\vKmax := \max_{\br\in \mathbb{S}^2}\; v(\br).
\]
\end{theorem}

\begin{proof}
Let $\lambda \in \R$ be an eigenvalue of $E$, with corresponding eigenvector
$\bq\in \mathbb{C}^{(L+1)^2}$.
Then from \eqref{eq:Eint},
\begin{align}\label{eq:xEx}
    \bq^* E \bq
    &= \sum_{j=0}^{L}\sum_{\mu=-j}^j\sum_{\ell=0}^L \sum_{m = -\ell}^\ell
     \overline{q_{j,\mu}}E_{j,\mu;\ell,m}q_{\ell,m}\\
    &= \sum_{j=0}^{L}\sum_{\mu=-j}^j\sum_{\ell=0}^L \sum_{m = -\ell}^\ell
     \overline{q_{j,\mu}} q_{\ell,m} \int_{\Sph} \overline{Y_{j, \mu}(\br)}Y_{\ell, m}(\br)
    v(\br)\, \dd S(\br)\nonumber\\
    &= \int_{\Sph} \overline{q(\br)} q(\br)
    v(\br)\dd S(\br),\nonumber
\end{align}
where
\begin{equation}\label{eq:z_J}
    q(\br) := \sum_{\ell=0}^L
    \sum_{m= -\ell}^\ell
    q_{\ell,m} Y_{\ell,m}(\br).
\end{equation}
It follows that
\begin{align}\label{eq:lambda_argument}
\lambda \|\bq\|_{\ell_2}^2 &=  \bq^*(\lambda
\bq) = \bq^* E \bq
=\int_{\Sph} |q(\br)|^2 \vK(\br)\, \dd S(\br)\\
&< \| q \|_{L_2}^2 \, \vKmax
= \sum_{j=0}^{J}\sum_{\mu = -j}^j |q_{j,\mu}|^2 \,\vKmax
= \|\bq\|_{\ell_2}^2\, \vKmax \nonumber,
\end{align}
where the inequality is strict because $\vK$,
being a spherical polynomial of non-zero degree,
cannot be identically equal to either its maximum or
minimum value.  Similarly, we have a lower bound
\[
\lambda \|\bq\|_{\ell_2}^2 >\|\bq\|_{\ell_2}^2\, \vKmin,
\]
together  proving $\lambda \in  (\vKmin, \vKmax)$.
\end{proof}

Note that even if the true (non-polynomial) mask  lies in $[0, 1]$ for all $\br\in\Sph$, the Gibbs phenomenon
will typically produce oscillations in $v$, making $\vKmin < 0$ and $\vKmax > 1$.

We also note in passing that $q$ is an eigenvector belonging to $\lambda$ for
the integral equation
\[
\int_{\mathbb{S}^2}{\mathcal K}_{L}(\br,\br') q(\br')  \,\dd \br'
= \lambda\, q(\br),
\]
where ${\mathcal K}_{L}(\br,\br')$ is the integral kernel given by
\begin{align}\label{eq:kernel}
{\mathcal K}_{L}(\br,\br')
 &= \sum_{\ell=0}^{L}\sum_{m = -\ell}^\ell
  Y_{\ell, m}(\br)\overline{Y_{\ell, m}(\br')} v(\br')\\
&= \sum_{\ell=0}^{L}\frac{2\ell+1}{4\pi}
  P_\ell(\br\cdot \br') v(\br'),\nonumber
\end{align}
and in the last step we used the addition theorem for spherical harmonics.
Here $P_\ell$ is the Legendre polynomial of degree $\ell$,
normalised so that $P_\ell(1) = 1$.

\section{Solving $E\ba = \ba^{v}$} \label{sec:Solving}
As with any ill-posed system, it is essential to build in
\emph{a priori} knowledge of the solution.
Neumaier~\cite[Section~8]{neumaier1998solving},
knowing that the true solution of an ill-posed problem is generally smooth, controls the smoothness
through a smoothing operator $S$. 
However, in this problem a smoothing operator would not be appropriate because the solution is the opposite of smooth, since for
each $\ell,m$ the unknown quantity $a_{\ell,m}$ is a realisation of an
independent random variable. That is a property we must build into the
solution. Accordingly, we assume, in accordance with the usual assumptions
for the CMB, that the $a_{\ell,m}$ are mean-zero uncorrelated random
variables with covariance $(C_\ell)_{\ell=0}^L$, where $C_\ell$ is real. Details on using Gaussian random fields to model the CMB can be found in the book by Marinucci and Peccati \cite{Marinucci_Peccati2011}.

We allow general $J$ in the range
$L \le J \le L+K$, giving a linear system with $(J+1)^2$ equations, so typically an over-determined linear system with more equations than unknowns, implying that an exact solution does not in general exist.

Moreover, we assume that the original field coefficients  $a_{\ell,m}$ are corrupted by noise, so the actual model is
\begin{equation}\label{eq:perturbed}
    E\ba_\beps = \ba^{v} + \beps^{v} = (\ba+\beps)^{v}
   ,
\end{equation}
where $\beps$ is a vector of independent mean-zero random variables
$\varepsilon_{\ell, m}$ with a diagonal covariance matrix $\eCov$, and $\ba_\beps$ is an approximation to $\ba$.
We also assume that the $\varepsilon_{\ell,m}$ and the $a_{\ell,m}$ are all statistically independent, so that in terms of expected values we have
\begin{align}\label{eq:expect_quad}
    &\inner{\varepsilon_{\ell,m}} = 0, \quad \inner{a_{\ell,m}} = 0 ,  \nonumber \\
    &\inner{\varepsilon_{\ell, m} \overline{a_{\ell',m'}}} = 0,
    \quad \inner{a_{\ell,m}\overline{a_{\ell',m'}}} = C_{\ell}\delta_{\ell, \ell'}\delta_{m,m'} \, ,\\
    &\inner{\varepsilon_{\ell, m} \overline{\varepsilon_{\ell',m'}}}
    = \eCov_{\ell,m} \delta_{\ell, \ell'}\delta_{m,m'} \nonumber.
\end{align}
We deduce the following
expectations of quadratic forms:
\begin{align}\label{eq:expect_q}
    \inner{\ba \ba^*} & = \aCov\nonumber\\
   \inner{\ba^{v} \ba^* }&= \inner{(E\ba)\ba^*} =  E \nonumber \aCov\\
    \inner{\ba^v(\ba^{v})^{*}} &= \inner{(E\ba)(\ba^{*}E^{*})}
    = E \aCov E^*\\
    \inner{\beps \ba^*} & = \bzero \nonumber \\
    \inner{\beps \beps^*} & = \eCov, \nonumber\\
  \inner{\beps^v (\beps^v)^*} & = E \eCov E^{*}, \nonumber
\end{align}
where 
\begin{equation}\label{eq:Cmat}
  \aCov_{\ell,m;\ell',m'} = C_{\ell}\delta_{\ell,\ell'}\delta_{m,m'} .
\end{equation}

Let $\Lambda\in\mathbb{C}^{(L+1)^2 \times (L+1)^2}$ be a real symmetric-positive definite matrix, with
associated norm $\| \ba \|_\Lambda = \left(\ba^* \Lambda
\ba\right)^\frac{1}{2}$ defined by
\begin{equation} \label{E:Norm}
\| \ba \|_\Lambda^2 = \ba^* \Lambda \ba =
\tr\left[ \ba \ba^* \Lambda\right] =
\tr\left[ \Lambda \ba \ba^* \right],
\end{equation}
where we used the matrix property $\mathrm{tr}(AB)= \mathrm{tr}(BA)$.
The following theorem gives a condition for minimising the expected squared $\Lambda$-norm error of an approximate solution of \eqref{eq:perturbed}.  It is an extension/specialisation of
\cite[Theorem~8]{neumaier1998solving}, which that author attributes to
\cite{Bertero1980}.

\begin{theorem}\label{thm:bagen}
Consider the over-determined linear system $E \ba_\beps = \ba^v + \beps^v$, where $\ba^v = E\ba$,
$\beps^v = E \beps$  and $\ba$ and $\beps$ have the stochastic properties in
\eqref{eq:expect_q}. Assume that the $(J+1)^2 \times  (L+1)^2$ matrix $E$ has full rank $(L+1)^2$. Among all  approximations of the form
{$\ba_\beps \approx Q(\ba^{v}+ \beps^v)$}, where $Q$ is a non-random $(L+1)^2 \times (J+1)^2$ matrix, the expected squared error $\inner{\|\ba-Q 
(\ba^v+\beps^v)\|_\Lambda^2}$ is minimized by any solution $\Qmin$ of the equation 
\begin{equation}\label{eq:Qmin}
\Qmin E (\aCov + \eCov) = \aCov.
\end{equation}
 The resulting minimum expected squared error is
\begin{equation}\label{eq:msegen}
    \inner{\| \ba-\amin \|_\Lambda^2} =
    \tr\left[\Lambda \bigl( \aCov -\aCov(\aCov+\eCov)^{-1}\aCov \bigr)\right],
\end{equation}
where $\amin:= \Qmin(\ba^v+\beps^v)$.
\end{theorem}

\begin{remark}
The minimizer $\amin$ is in general not unique.
\end{remark}

\begin{proof}
Writing $\by_\beps:= \ba^v + \beps^v$,
a general linear approximation can be written as 
$Q\by_\beps = \Qmin \by_\beps + R \by_\beps$, where $\Qmin$ is an as yet unknown minimizer, and $R = Q - \Qmin$ is a matrix in
$\mathbb{C}^{(L+1)^2 \times (J+1)^2}$.
The mean square error can now be written as 
\begin{align*}
  \|\ba-(\Qmin +R)\by_\beps\|_\Lambda^2
   & = \|\ba-\Qmin \by_\beps\|_\Lambda^2 + \|R\by_\beps\|_\Lambda^2
     - 2\Re\left[\left(\ba-\Qmin \by_\beps)^{*}\right) \Lambda R\by_\beps\right]\\
    & = \|\ba-\Qmin \by_\beps\|_\Lambda^2 + \|R\by_\beps\|_\Lambda^2
      - 2\Re\left[\tr\left(\Lambda R\by_\beps(\ba-\Qmin \by_\beps)^{*}\right)\right].
\end{align*}
On taking expected values and using \eqref{eq:expect_q}
we have
\begin{align*}
\inner{\tr  (\Lambda R \by_\beps (\ba - \Qmin \by_\beps)^*} &= 
\tr  \left(\Lambda R \inner{\by_\beps (\ba - \Qmin \by_\beps)^*} \right) \\
&=\tr  \left(\Lambda R \inner{ (\ba^{v} + \beps^v) (\ba - \Qmin (\ba^v + \beps^v))^*} \right) \\
&=\tr \left( \Lambda R \left(E \inner{\ba \ba^*} - E\inner{\ba \ba^* + \beps \beps^*}\
 E^* \Qmin^* \right)\right)\\
&=\tr \left( \Lambda R \left( E \aCov - E(\aCov  + \eCov) E^* \Qmin^* \right) \right).
\end{align*}
So
\begin{equation*}
\begin{split}
    \inner{\|\ba-(\Qmin +R)\by_\beps\|_\Lambda^2}
    &= \inner{\|\ba-\Qmin \by_\beps\|_\Lambda^2} + \inner{\|R\by_\beps\|_\Lambda^2} \\
    &\quad \qquad
    -2\Re  \tr \left( \Lambda R \left( E \aCov - E(\aCov  + \eCov) E^* \Qmin^* \right) \right) .
\end{split}
\end{equation*}
By definition, $\Qmin $ is a minimizer of 
$\inner{\| \ba - Q {\by_\beps}\|_\Lambda^2}$, so the linear term must vanish for all $R$.  More precisely, we must
have
\begin{equation}\label{eq:EOmega}
E \aCov  = E  (\aCov   + \eCov) E^* \Qmin ^*, \quad \mbox{or equivalently} \quad \Qmin E(\aCov +\eCov)  E^* = \aCov E^*,
\end{equation}
since otherwise by taking $R$ to be 
$\bigl(E \aCov   - E  (\aCov  + \eCov)E^* \Qmin ^*\bigr)^*$ we obtain a contradiction. It is easily seen that the second equality in \eqref{eq:EOmega} is equivalent to \eqref{eq:Qmin}: starting with \eqref{eq:Qmin} by right multiplying by $E^*$, starting with \eqref{eq:EOmega} by right multiplying by $E$ and using the invertibility of $E^* E$, noting that $E^*E$ is a square matrix of full rank $(L+1)^2$.  
If \eqref{eq:Qmin} holds then we have

\begin{align*}
    \inner{\| \ba-(\Qmin +R)\by_\beps\|_\Lambda^2}
    & = \inner{\| \ba-\Qmin \by_\beps \|_\Lambda^2} +
        \inner{\| R\by_\beps\|_\Lambda^2}\\
    & \geq \inner{\| \ba-\Qmin \by_\beps \|_\Lambda^2}
\end{align*}

with equality for $R = 0$, corresponding to
$\amin  = \Qmin \by_\beps$.

The expected squared error is
\begin{align*}
  \inner{\|\ba-\amin \|_\Lambda^2} & =
  \inner{\| \ba-\widehat{Q}\by_\beps \|_\Lambda^2}\\
  &= \inner{\bigl(\ba-\widehat{Q}\by_\beps\bigr)^* \Lambda \bigl(\ba-\Qmin\by_\beps\bigr)}\\
  &= \tr\left[ \Lambda \inner{\bigl(\ba-\Qmin\by_\beps\bigr) \bigl(\ba-\Qmin\by_\beps\bigr)^*}\right]\\
  &= \tr\left[\Lambda\left(\inner{\ba\ba^*}+\inner{\Qmin\by_\beps\by_\beps^{*}\Qmin^*}-2\Re\bigl(\Qmin\inner{\by_\beps\ba^*}\bigr)\right)\right]\\
  &= \tr\left[\Lambda\left( \aCov+\widehat{Q} (E \aCov E^*+ E\eCov E^*)\widehat{Q}^* - 2\Re\bigl(\widehat{Q}E \aCov\bigr)\right)\right],
    \end{align*}

Now by \eqref{eq:EOmega}
\begin{equation*}
    \Qmin E \aCov = \Qmin(E \aCov E^*+ E \eCov E^*)\Qmin^*,
\end{equation*}
and hence
\begin{align*}
\tr[ \Lambda \Qmin E \aCov]
&= \tr [\Lambda \Qmin E (\aCov +\eCov) E^*\Qmin^*]\\
&= \tr[(\aCov +\eCov)^{1/2}E^*\Qmin^*\Lambda \Qmin E (\aCov +\eCov)^{1/2}]\\
&= \|\Qmin E (\aCov +\eCov)^{1/2}]\|_{\Lambda}^2,
\end{align*}
which is real, implying
\begin{align*}
  \inner{\|\ba-\amin \|_\Lambda^2}
  &= \tr\left[\Lambda\left( \aCov-\widehat{Q}E \aCov \right)\right]\\
  &=\tr\left[\Lambda\left( \aCov- \aCov(\aCov +\eCov)^{-1}\aCov \right)\right].\\
\end{align*}

\end{proof}
 
\begin{remark}
Note that the equation \eqref{eq:EOmega} determining the minimizer $\widehat{Q}\by_\beps$ of the expected mean-square error
does not depend on the matrix $\Lambda$, i.e. on the choice of
quadratic norm. For example, using $\Lambda = I$ or $\Lambda =  \aCov$
does not change $\Qmin $.
\end{remark}

\begin{corollary}
Under the conditions of Theorem  \ref{thm:bagen}, let $\Gamma$ be an arbitrary positive definite matrix of size $(J+1)^2 \times (J+1)^2$.  Then a vector $\widehat{\ba}\in \mathbb{R}^{(L+1)^2}$ that achieves the minimal error given in \eqref{eq:msegen} is 
\begin{equation}\label{eq:afromalpha}
\widehat{\ba} := \aCov(\aCov + \eCov)^{-1}\balpha,
\end{equation}
where $\balpha\in \mathbb{R}^{(L+1)^2}$ is the unique solution of
\begin{equation}\label{eq:alpha}
E^* \Gamma E \balpha = E^* \Gamma (\ba^v + \beps^v).
\end{equation}
\end{corollary}

\begin{proof}
The matrix $\Qmin$ defined by 
\begin{equation}\label{eq:Qmin_example}
\Qmin :=  \aCov (\aCov+\eCov)^{-1} (E^* \Gamma E)^{-1} E^* \Gamma
\end{equation}
is easily seen to satisfy the condition 
\eqref{eq:Qmin} in Theorem \ref{thm:bagen}.
Equally, it is easily seen that the corresponding minimizer 
\[
\widehat{\ba} := \Qmin(\ba^v + \beps^v)
\]
can be written exactly as stated in the corollary.
\end{proof}

\begin{remark}

The corollary gives our prescription for computing the coefficient vector $\amin$. 
Note that the postprocessing step in \eqref{eq:afromalpha} is easily carried out
given that the matrices $\aCov$ and $\eCov$ are diagonal,
since each element of $\alpha$ is by this step merely reduced by a known factor.
Note also that equation \eqref{eq:alpha} is just the normal equation for the linear system if, as we shall assume in practice, $\Gamma$ is the identity matrix.  Formation of the normal equations can greatly increase the condition number of an already ill-conditioned system.  In practice we shall address the ill-conditioning either by QR factorisation of the matrix $E$, or (less desirably) by adding a regularising term to the right-hand side, to obtain
\begin{equation}\label{eq:alphareg}
(E^* \Gamma E  + \Sigma) \balpha = E^* \Gamma (\ba^v + \beps^v),
\end{equation}
where $\Sigma$ is an empirically chosen positive definite  $(L+1)^2 \times (L+1)^2 $ matrix.     
\end{remark}

The following proposition shows that if the elements of $\ba^v$ and $\beps^v$ have the correct symmetry for real-valued fields $a^v = a v$ and $\epsilon^v = \epsilon v$, then 
the computed values of $\widehat{\ba}$ also follow the same symmetry. The practical importance of this result is that the symmetry property, since it occurs naturally, does not need to be enforced.

\begin{proposition}
Assume that the components of $\ba^v$ and  $\beps^v$ satisfy
\[
  a^v_{j,\mu}= (-1)^\mu \overline{a^v_{j,-\mu}}
  \quad\mbox{and}\quad
  \varepsilon_{j,\mu}^v = (-1)^\mu \overline{\varepsilon}_{j,-\mu}^v,
  \quad \mu = -j, \ldots j, \; j\ge 0.
\]
  Assume also that the positive definite matrices $\aCov, \eCov$ and $\Gamma$, and also $\Sigma$ if present, are all diagonal, and that their diagonal elements are positive numbers independent of the second label $\mu$ or $m$.  Then $\amin$ given by \eqref{eq:alpha} and \eqref{eq:afromalpha} satisfies 
\[
\amin_{\ell,m} = (-1)^m \overline{\amin_{\ell, -m}},\quad
m = -\ell,
\ldots,\ell,\; \ell \ge 0.
\]
\end{proposition}

\begin{proof}
We first show that the components of $\bb:= E^* \Gamma(\ba^v +\beps^v)$ satisfy
\[
b_{\ell,m} = (-1)^m\overline{b_{\ell,-m}}, \quad m = -\ell,\ldots,\ell,\; \ell \ge 0.
\]
We have, using \eqref{eq:Enegm}, 
\begin{align*}
   \overline{b_{\ell,-m}} &=  
   \overline{
   \sum_{j}\sum_\mu (E^*)_{\ell,-m;j,\mu} \Gamma_j (a^v_{j,\mu}+\varepsilon_{j,\mu}) }\\
   &=  \sum_{j}\sum_\mu (-1)^{m - \mu}(E^*)_{\ell,m;j,-\mu} \Gamma_j
   \overline{(a^v_{j,\mu}+\varepsilon_{j,\mu}})\\
   &= (-1)^m\sum_{j}\sum_\mu (E^*)_{\ell,m;j,-\mu} \Gamma_j (a^v_{j,-\mu}+
   \varepsilon_{j,-\mu})\\
   &= (-1)^m  b_{\ell,m},
\end{align*}
as required.   A similar argument shows that 

\[
(E^* \Gamma E +  \Sigma)_{\ell,-m; \ell', m'}
=(-1)^{m-m'}\overline{(E^* \Gamma E + \Sigma)_{\ell,m; \ell', -m'}}.
\]

Since $\balpha$ is the unique solution of 
\[
(E^* \Gamma E + \Sigma) \balpha = \bb,
\]
by taking the $(\ell, -m)$ component of this equation we obtain
\[
\sum_{\ell'}\sum_{m'}(E^* \Gamma E + \Sigma)_{\ell, -m ; \ell' m'} \alpha_{\ell',m'} 
=b_{\ell, -m},
\]
which with the above symmetry properties leads to 
\[\sum_{\ell'}\sum_{m'}(-1)^{m-m'}\overline{(E^* \Gamma E + \Sigma)_{\ell, m ; \ell' -m'} }\alpha_{\ell',m'} 
=(-1)^m \overline{b_{\ell, m}};
\]
On taking the complex conjugate and dividing by $(-1)^m$ this gives us
\begin{equation}\label{eq:c_eqn}
(E^* \Gamma E + \Sigma) \bc = \bb,
\end{equation}
where 
\begin{equation}\label{eq:c_symm}
c_{\ell', -m'} :=  (-1)^{m'} \overline{ \alpha_{\ell', m'}},\quad m' = -\ell',\ldots,\ell',\; \ell' \ge 0. 
\end{equation}
We see by uniqueness of the solution of \eqref{eq:c_eqn} that $\bc =\alpha$, thus by \eqref{eq:c_symm} the vector $\balpha$ has the desired symmetry. Multiplication by $\aCov (\aCov + \eCov)^{-1}$ clearly preserves the symmetry, thus the proof is complete.
\end{proof}

\section{Axially symmetric masks}\label{sec:axsym}
A general mask $v(\br)$ leads to a large dense matrix $E$, of size $(J+1)^2 (L+1)^2 \times (J+1)^2 (L+1)^2 $, see \eqref{eq:Edefined}, a size beyond present resources if $J$ and $L$ are in the hundreds.
In this section we consider the more tractable special case in which the mask $v$ is axially symmetric, i.e.
\[
v(\br) = v(\theta, \phi) = v(\theta),
\]
with $v$ being a function of the polar angle $\theta$ and independent of the azimuthal angle  $\phi$. For this case we have
\[
v_{k,\nu} = w_{k } \delta_{\nu,0},
\]
where
\begin{equation}\label{eq:w}
w_{k}:= v_{k,0} = \int_{\Sph} v(\br)\overline{Y_{k,0}(\br)}
\dd S(\br)
= \sqrt{\pi (2 k + 1)} \int_0^\pi v(\theta) P_{k}(\cos(\theta))\sin\theta \dd\theta.
\end{equation}
Note that $w_{k}$ is real, and that
\[
w_{k}= 0 \quad\mbox{if }k \mbox{  is odd and also }
v(-\br) = v(\br).
\]
In this case it follows from \eqref{eq:Dprops} 
that $E_{j,\mu;\ell,m}=0$
 unless $\mu = m$.
 Thus it is convenient to introduce a new notation,
 \begin{equation}\label{eq:Emlj}
 E_{j,\ell}^{(m)} := E_{j,m;\ell,m}.
 \end{equation}
Equation \eqref{eq:aequalsEa} now becomes
\begin{equation}\label{eq:axial}
\sum_{\ell=0}^L E_{j,\ell}^{(m)}a_{\ell,m}= a_{j,m}^v,
\end{equation}
in which the coefficients belonging to different values of $m$ are
completely decoupled.  This can be seen as just a special case of
\eqref{eq:aeqav_math}, albeit with uncoupled values of $m$, thus all of
the analysis in Sections~\ref{sec:PropE} and \ref{sec:Solving} remains applicable.

Note that from \eqref{eq:Ereal} of Lemma~\ref{lem:Eprop},
$E^{(m)}_{j,\ell}$ is real and
symmetric, $E^{(m)}_{j,\ell}= E^{(m)}_{\ell,j}$.  Moreover
\[
 E^{(m)}_{j,\ell} = 0 \,\mbox{ if } j + \ell \,\mbox{ is odd and }
\, v(-\br) = v(\br) \mbox{, or if } \ell < |m| \mbox{ or }j < |m| .
 \]

We can treat $E^{(m)}_{j,\ell}$ as an $(J-|m|+1) \times (L-|m|+1)$ matrix,
but how should we choose $J$?  The choice $J = L$ inevitably leads to a poorly conditioned linear system.  There would seem to be considerable benefit, at
least in theory, in taking the largest value $J = L + K$, to ensure that
the resulting overdetermined linear system makes use of all available
information.

The equation to be solved in practice is, instead of \eqref{eq:alpha}, now
\begin{equation}\label{eq:eqtosolveaxial}
  \bigl( (E^{(m)})^* \Gamma_a E^{(m)}\bigr) \balpha  = (E^{(m)})^* \Gamma_a (\ba^v + \beps^v);
\end{equation}
Or if regularisation is desired, then 
instead of \eqref{eq:alphareg} the equation to be solved becomes
\begin{equation}\label{eq:eqtosolveaxialReg}
  \bigl( (E^{(m)})^* \Gamma_a E^{(m)} +  \Sigma_a\bigr) \balpha  = (E^{(m)})^* \Gamma_a (\ba^v + \beps^v).
\end{equation}
Here $\Gamma_a$ and $\Sigma_a$ have the same diagonal values as  $\Gamma$ and $\Sigma$, but the second label on rows and columns has now disappeared, and the new matrices are of size $J \times J$ and $L \times L$ respectively.

\section{Numerical experiments} \label{sec:NumExp}

Recall that our goal is to reconstruct a scalar random field on the sphere
given only 
a masked and noisy version of the field.
We have seen in previous sections  that the problem can be reduced to the solution of the overdetermined linear system 
\begin{equation}\label{eq:model}
  E \ba = \bb^v , 
\end{equation}
where the given data $\bb^v = (\ba + \beps)^v = \ba^v + \beps^v$ are the spherical harmonic coefficients of the masked noisy map
with the coefficients $\ba$  corrupted by independent Gaussian noise $\beps$ with
mean $\inner{\beps} = \bzero$ and
variance $\inner{\beps \beps^*} = \eCov$.


To illustrate the potential of the method we consider numerical experiments
where we know the ``true'' solution $\ba$,
so we can calculate errors to test performance and the effect of model parameters,
including taking $\beps = \bzero$.
\begin{figure}[ht]
    \centering
    \includegraphics[scale=0.6]{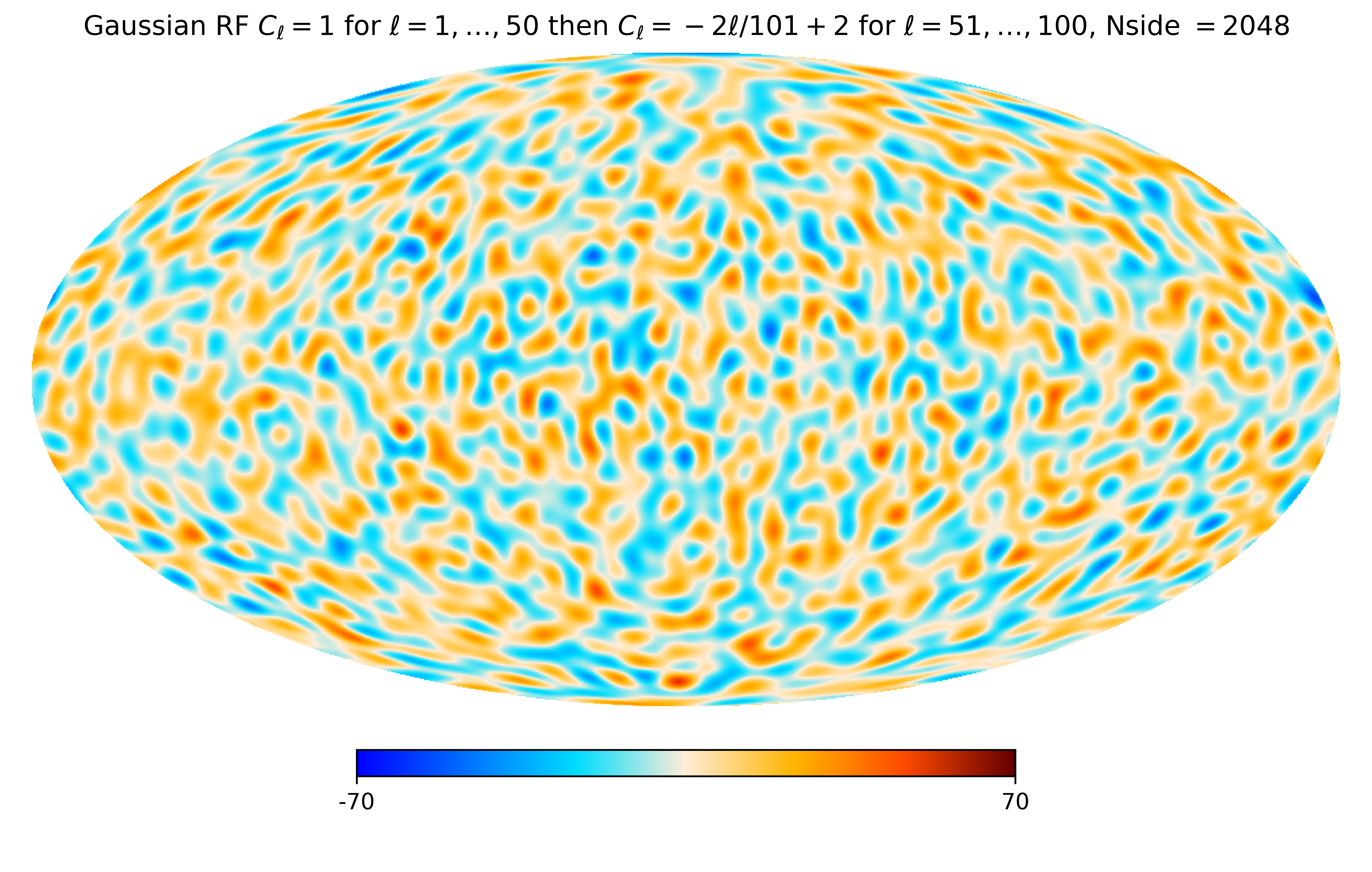}
    \caption{Original Gaussian random field}
    \label{fig:originalRF}
\end{figure}

In the experiments a specified angular power spectrum $C_\ell, \ell = 2,\ldots,L$ 
is used to generate an instance of a Gaussian random field with
known spherical harmonic coefficients $\ba$ at the
$N_{\textrm{pix}} = 50,331,648$ \texttt{HEALPix}\footnote{\href{http://healpix.sf.net}{\texttt{http://healpix.sf.net}}}~\cite{2005ApJ...622..759G}
points ($N_{\textrm{side}} = 2048$), using the \texttt{HealPy}\footnote{\href{https://pypi.org/project/healpy/}{\texttt{https://pypi.org/project/healpy/}}} package~\cite{Zonca2019}.
Noise is then added as described in Subsection \ref{ss:Experiments}.
The mask is then applied pointwise to the noisy map, and the masked noisy map used to calculate the spherical harmonic coefficients $\bb^v$, again using the \texttt{HealPy} package.  The next step is to estimate the original Fourier coefficients $\ba_{\ell,m}$ using \eqref{eq:afromalpha} together with \eqref{eq:alpha}, or alternatively using \eqref{eq:afromalpha} with the regularised equation \eqref{eq:alphareg}.  The final step is to reconstruct the target field from its Fourier coefficients.

We consider a Gaussian random field with the artificial angular power spectrum
\[
C_\ell = g\left(\frac{\ell}{L+1}\right),
\]
where
\[
g(x) = \begin{cases}
 1 & \text{ for } 0 \le x \le 1/2 \\
 -2x + 2 & \text{ for } 1/2 \le x \le 1.
\end{cases}
\]
In the experiments, we assume that $L=100$ and $K=900$.
A realisation of the random field is shown in Figure~\ref{fig:originalRF}.  We shall use this realisation as the target field $a(\br)$ in all the following experiments.

\subsection{An axially symmetric mask} \label{sec:axsymact}

When the mask applied to the noisy data is axially symmetric,
the problem decomposes into independent problems
for each value of $m$, as in Section~\ref{sec:axsym}.
To construct the mask we first define the following non-decreasing function $p\in C^3(\R)$:
\begin{equation} \label{eq:maskp}
p(x)=
\begin{cases}
   0 & \text { for } x\le 0,\\
   x^4(35-84x+70x^2-20x^3) &  \text{ for } 0<x<1,\\
   1 & \text{ for } x \ge 1.
 \end{cases}
\end{equation}
Then as a function of the Cartesian coordinate $z \in [-1, 1]$ of a point
on the sphere, our mask is, for $0 < a_z < b_z < 1$
\begin{equation} \label{eq:maskz}
  v(z) = 
   p\left(\frac{|z|-a_z}{b_z-a_z}\right). 
\end{equation}

In Figure~\ref{fig:mask_kappa3}, an axially symmetric mask with
$a_z = \frac{\pi}{2} - \frac{10\pi}{180}$ and
$b_z = \frac{\pi}{2}- \frac{20\pi}{180}$
is plotted on the unit sphere.
\begin{figure}[ht]
    \centering
    \includegraphics[scale=0.55]{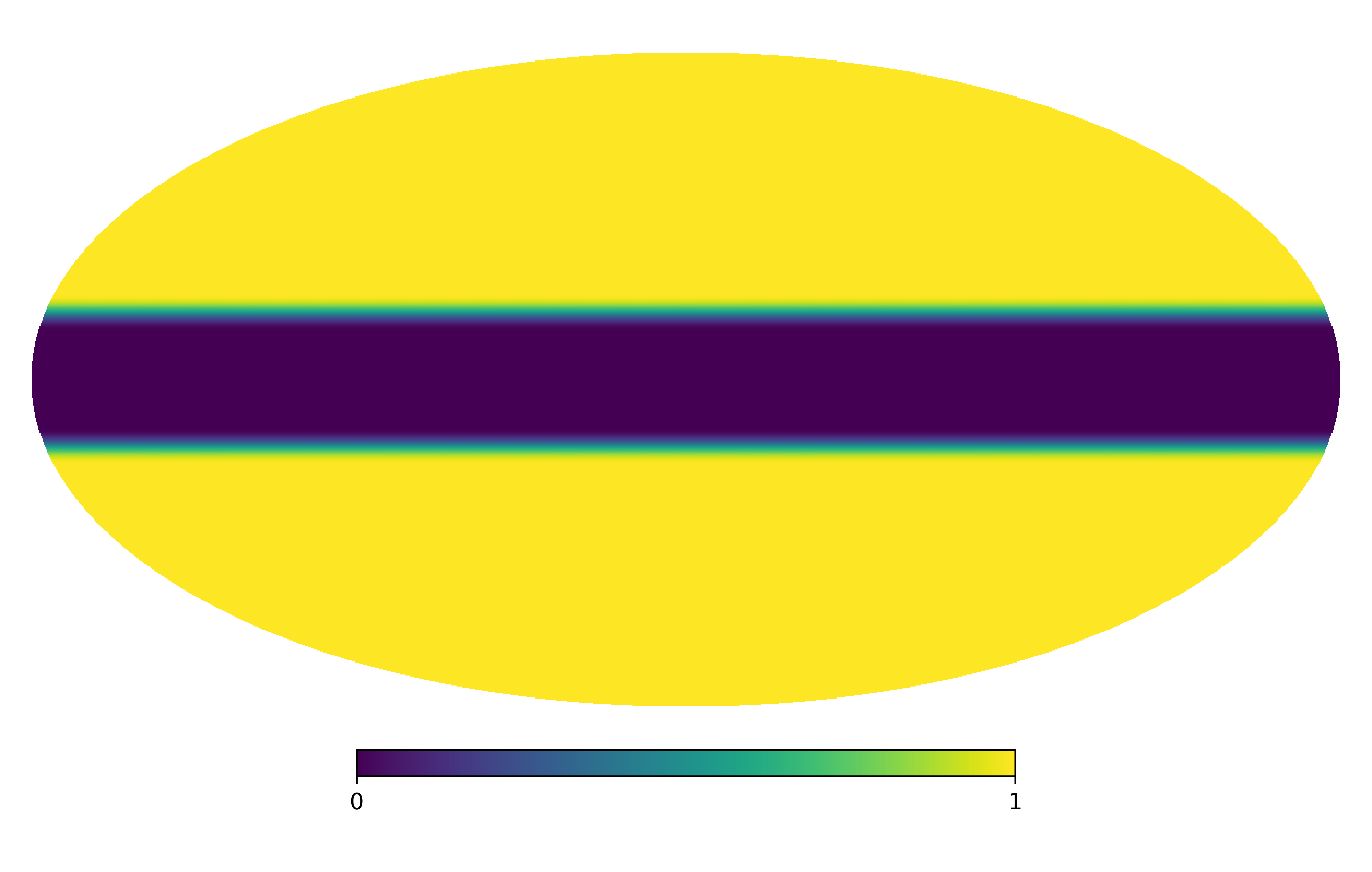}
    \caption{An axially symmetric $C^3$ mask with
    $a_z = \frac{\pi}{2} - \frac{10\pi}{180}$ and
    $b_z = \frac{\pi}{2}- \frac{20\pi}{180}$.
    }
    \label{fig:mask_kappa3}
\end{figure}
This mask has the value $1$ (and hence has no masking effect) for points on the sphere more than $20^{\degree}$ from the equator, the value $0$ (complete masking) within $10^{\degree}$ from the equator, and smooth variation in between, through the function $p$, see \eqref{eq:maskp}, with an argument expressed as a function of the $z$ coordinate of a point on the sphere.  (We do not use a discontinuous mask because a discontinuous function has slow convergence of its Fourier series, leading to Gibbs' phenomenon for the truncated Fourier series.)


\begin{figure}[ht]
    \centering
    \includegraphics[scale=0.80]{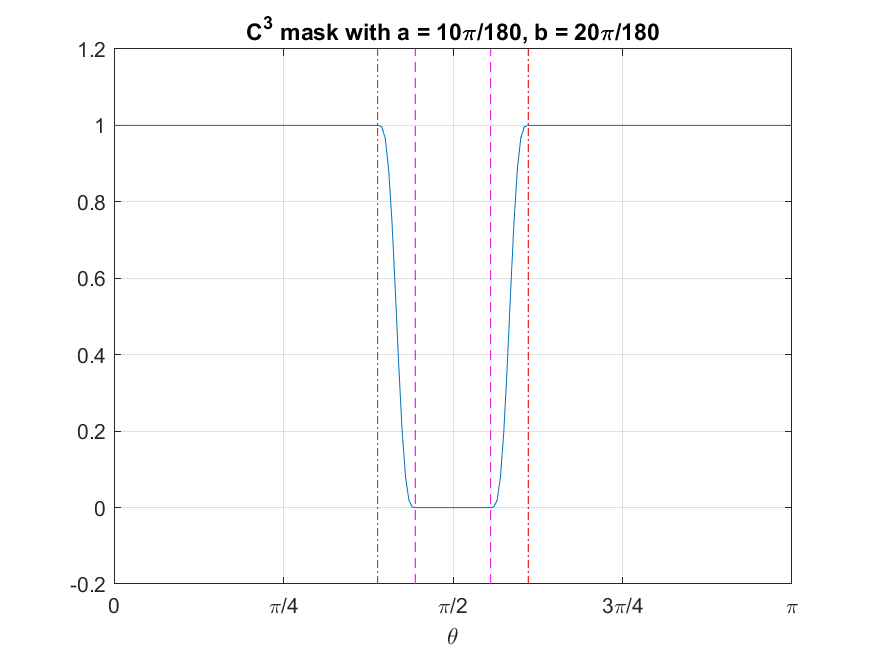}
    \caption{The $C^3$ mask $v(\theta,\phi)$ with
    $a = \frac{10\pi}{180}$, $b = \frac{20\pi}{180}$.}
    \label{fig:mask_theta}
\end{figure}
The transformation $z := \cos(\theta) = \cos(\frac{\pi}{2} - \varphi)$, with latitude $\varphi = \frac{\pi}{2} - \theta\in [-\frac{\pi}{2}, \frac{\pi}{2}]$, gives
\[
  v(\varphi, \phi) = 
   p\left(\frac{|\cos(\frac{\pi}{2} - \varphi)|-a_z}{b_z-a_z}\right) .
\]
In terms of latitude the transition region is $|\varphi| \in [a, b]$
where $a = \frac{10\pi}{180}$ and $b = \frac{20\pi}{180}$,
as illustrated in Figure~\ref{fig:mask_theta}.

The masked field, 
including the addition of noise as described in Section~\ref{ss:Experiments},
is illustrated in Figure~\ref{fig:maskedRF}.
\begin{figure}[ht]
    \centering
    \includegraphics[scale=0.6]{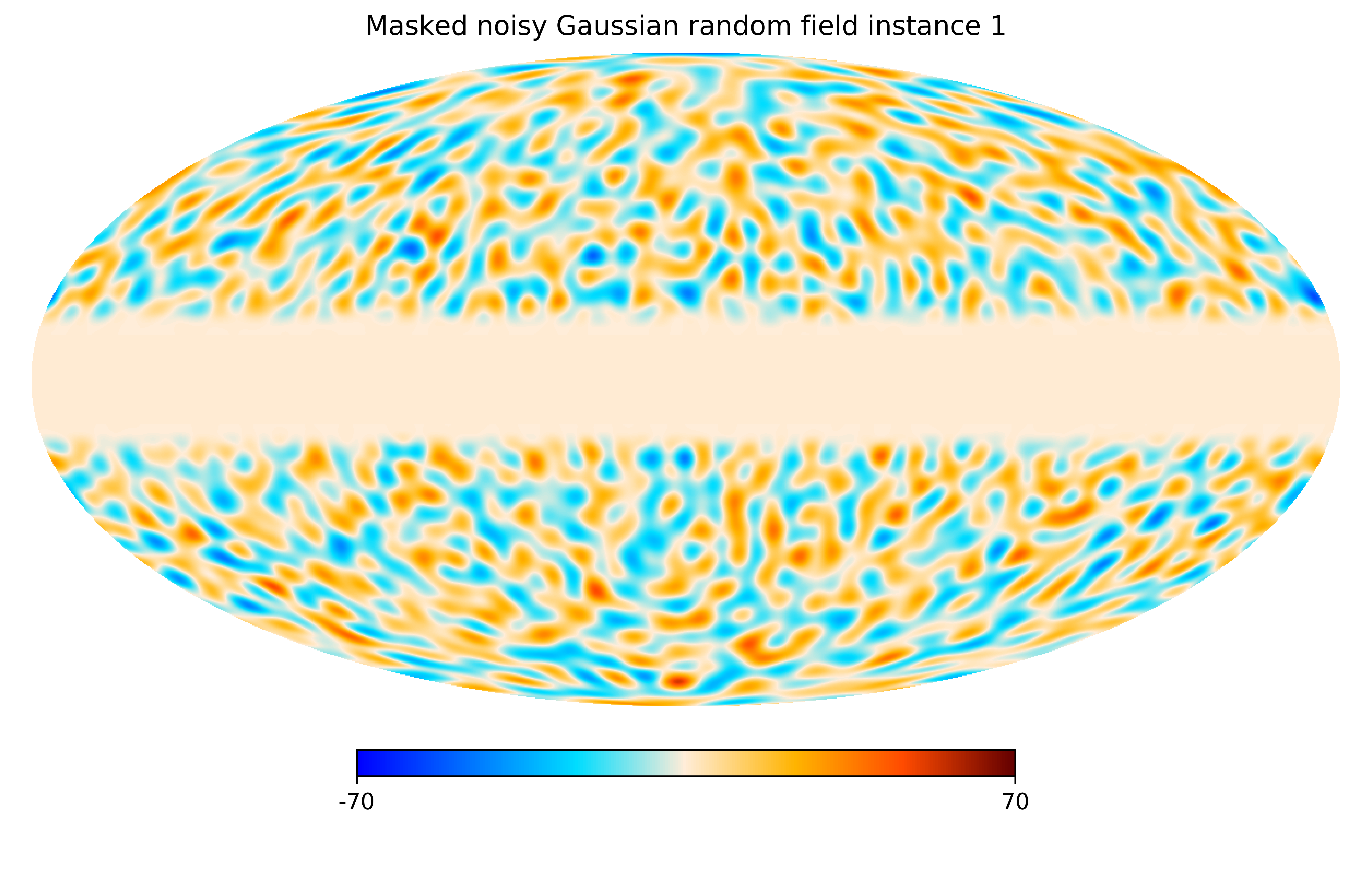}
    \caption{The masked noisy random field for $\tau = 10^{-4}$}
    \label{fig:maskedRF}
\end{figure}
 
Using this axially symmetric mask all the rectangular matrices $E=E^{(m)}$ described in Section~\ref{sec:axsymact} for $m=0,\ldots,L$ of sizes $(J+1-m) \times (L+1-m)$ with $J=L+K$ are pre-computed in parallel. Here $K$ is the maximum multipole in the spherical harmonics approximation of $v$ as in \eqref{eq:v}. We used the \texttt{sympy}
 package \cite{sympy}
to compute the entries for the matrix $E$. Fast quadrature methods on the unit sphere \cite{Plonka_et_al2023} could be used in a future implementation. 

\subsection{Numerical condition of the problem}
\label{sec:NumCond}

\begin{figure}[ht]
    \centering
    \begin{subfigure}{0.46\textwidth}
    \includegraphics[width=\textwidth,trim=1.0cm 0cm 1.1cm 0cm,clip]
    {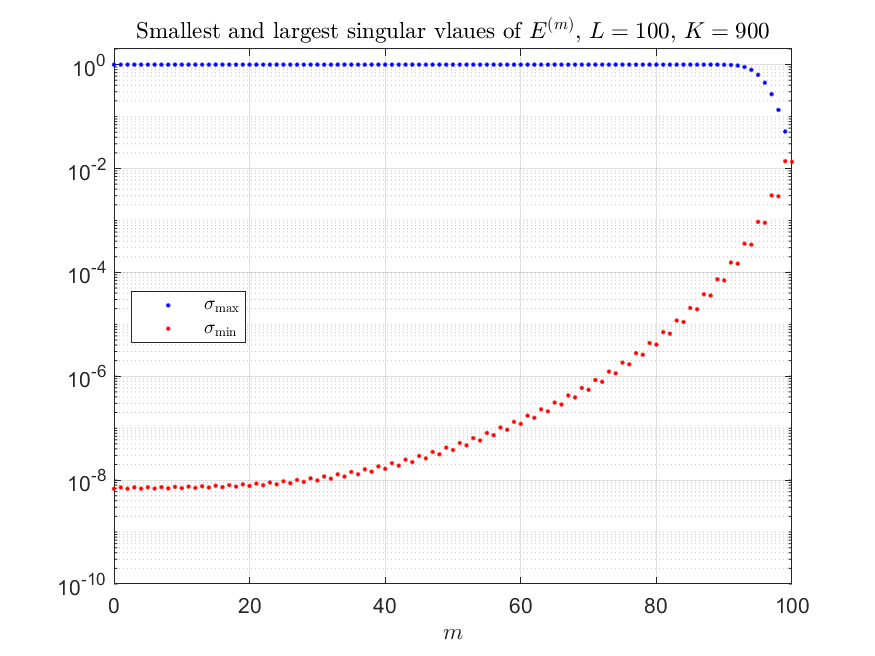}
    \caption{Smallest and largest singular values}
    \label{fig:svEmL100K900}
    \end{subfigure}
    \;
    \begin{subfigure}{0.46\textwidth}
    \includegraphics[width=\textwidth,trim=1.0cm 0cm 1.1cm 0cm,clip]
    {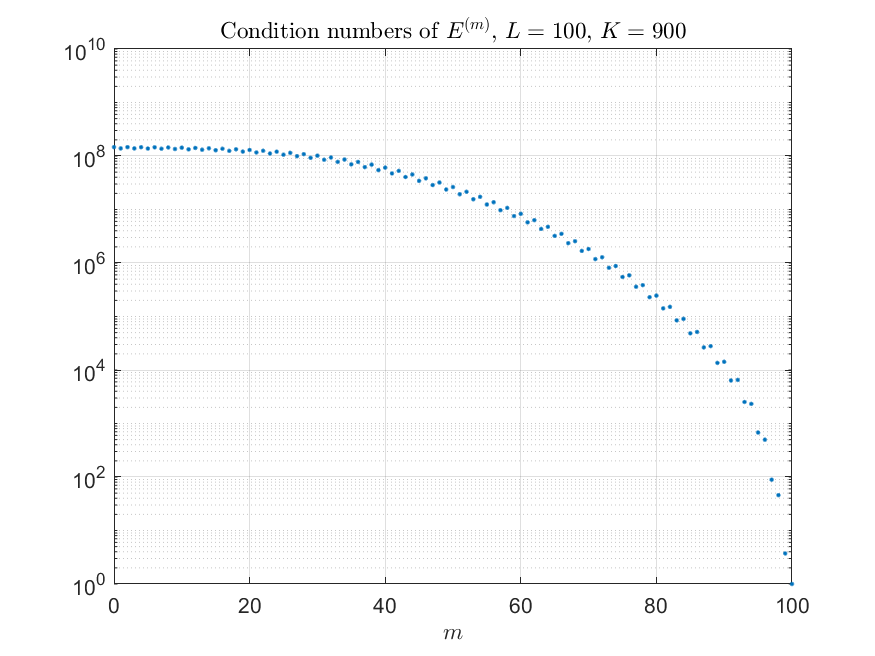}
    \caption{Condition numbers}
    \label{fig:condEmL100K900}
    \end{subfigure}    
    \caption{Singular values and condition numbers for the the matrices $E^{(m)}$ with $L=100$ and $K=900$ and the radially symmetric mask in Figure~\ref{fig:mask_theta}.}
    \label{fig:EmL100K900}
\end{figure}
Figures~\ref{fig:svEmL100K900} and \ref{fig:condEmL100K900} illustrate the
condition of the matrices $E^{(m)}$ for
$L = 100$, $K = 900$ and the mask in Figure~\ref{fig:mask_theta}.
The largest singular values in Figure~\ref{fig:svEmL100K900} are consistent with 
the upper bound on the singular values in Theorem~\ref{thm:svbnd}. 
Even though the values of the mask lie in $[0, 1]$,
the polynomial approximation of the mask
may have values slightly outside this interval due to the Gibbs phenomenon.
The ill-conditioning of the matrices $E^{(m)}$,
as illustrated in Figure~\ref{fig:condEmL100K900},
which can be severe especially for small $m$,
arises from the smallest singular value.

To avoid the squaring of the condition number of $E$,
as happens when solving the equations \eqref{eq:alpha}
with coefficient matrix $E^* \Gamma E$,
we can use the well-known QR factorization.
Consider the generic case where $E$ is a $(J+1)^2 \times (L+1)^2$ matrix.
Assume that the positive definite matrix $\Gamma$
has a readily available factorization
$\Gamma = \Theta^* \Theta$
(for example if $\Gamma$ is diagonal) and let
\begin{equation}\label{eq:Q1R1}
  \Theta E  = Q_1 R_1,
\end{equation}
where
$Q_1$ is a $(J+1)^2 \times (L+1)^2$ unitary matrix and
$R_1$ is an $(L+1)^2 \times (L+1)^2$ upper triangular matrix.
The solution to equation \eqref{eq:alpha},
assuming $R_1$ is non-singular (i.e. $\Theta E $ has full column rank), is then obtained by solving
\begin{equation}\label{eq:QReq}
  R_1 \balpha = Q_1^* \Theta (\ba^v + \beps^v)
\end{equation}
by back substitution.
On the other hand, if $R_1$, which has the same condition number as $\Gamma^{\frac{1}{2}} E$,
has diagonal elements which are too small, 
the regularized equation \eqref{eq:alphareg}
can be used.

\subsection{Experiments}\label{ss:Experiments}


Initially we assume there is no noise, that is $\beps=0$.
In this case the noise covariance matrix $\eCov$ is zero.  We 
 can still interpret \eqref{eq:model} as a least squares problem,
which is solved using the QR factorization.
The reconstructed field $\widehat{a}(\br), \br\in \Sph$
and the corresponding error field $\widehat{a}(\br) - a(\br), \br\in \Sph$
are shown in Figure~\ref{fig:No noise}.
The rows of Tables~\ref{tab:l2_errors} and \ref{tab:l2_errors 0 and 1}
with $\eCov_\ell = 0$ correspond to the no-noise experiment.
\begin{figure}[!ht]
    \centering
    \begin{subfigure}{\textwidth}
\hspace{2cm}
    \includegraphics[scale=0.55]{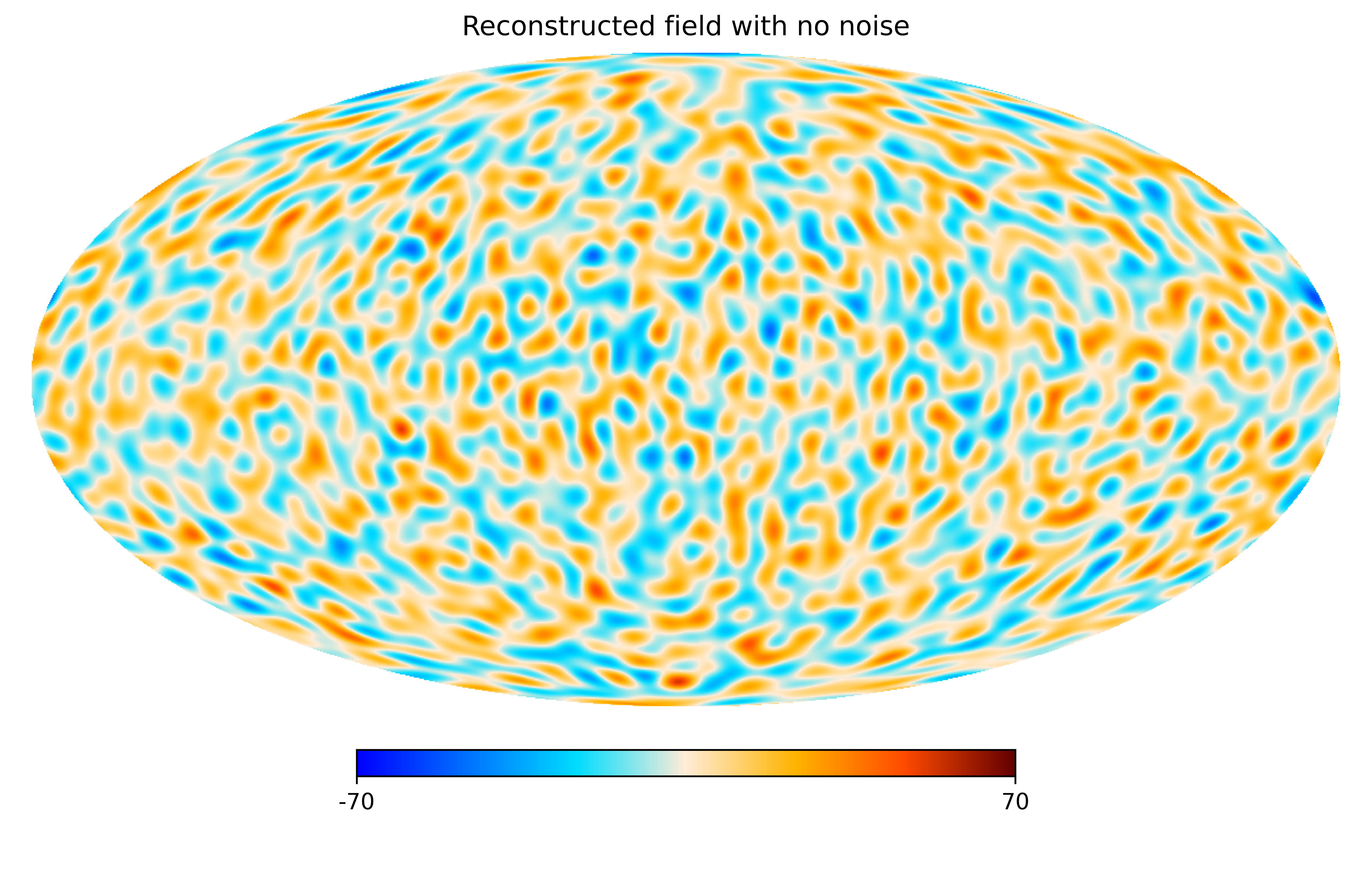}
    \caption{Reconstructed field}
    \label{fig:reconstructedRF_QR}
    \end{subfigure}
    \\
    \begin{subfigure}{\textwidth}
\hspace{2cm}
    \includegraphics[scale=0.55]{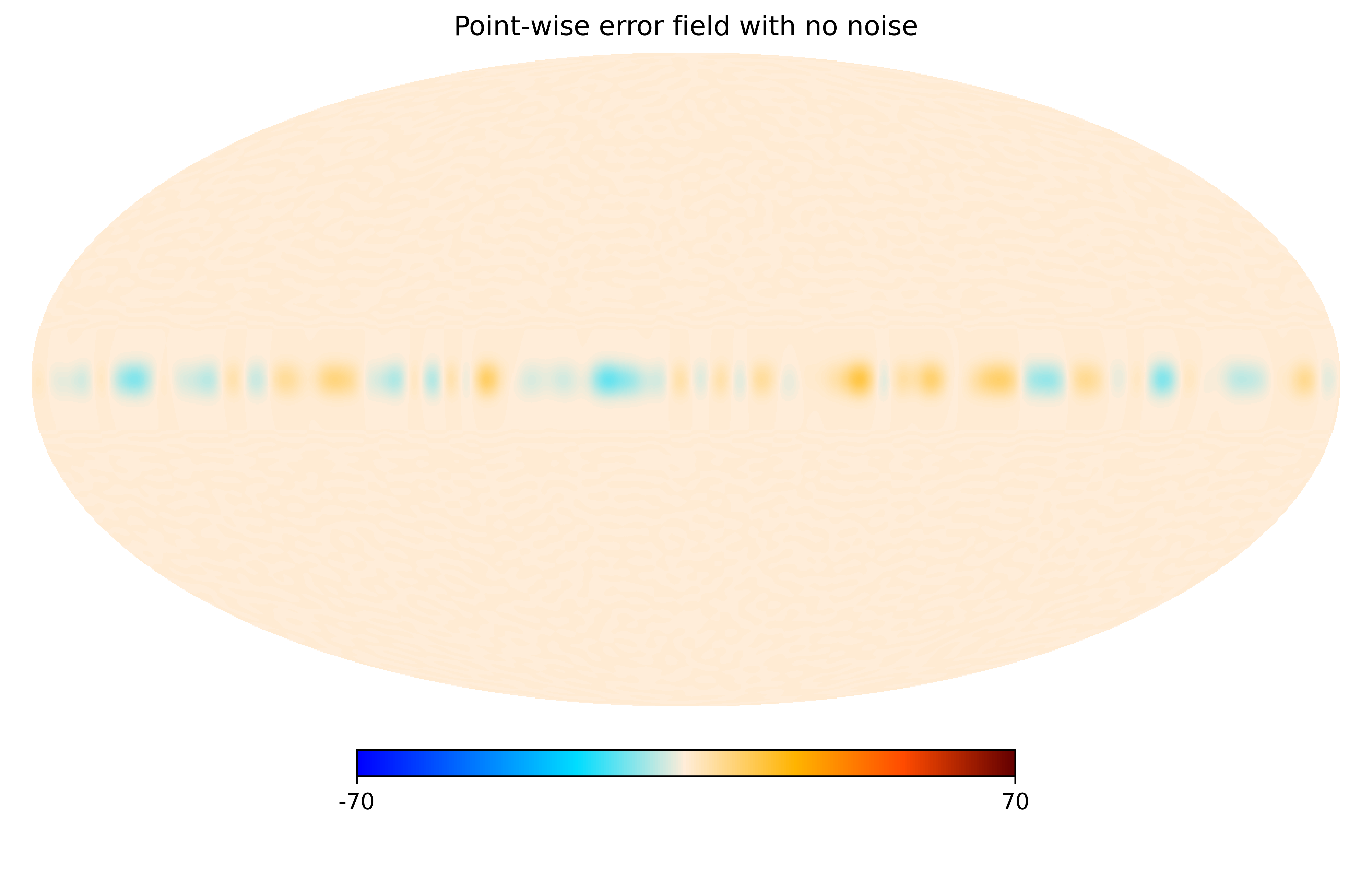}
    \caption{Error field}
    \label{fig:reconstructedRF_QR_error}
    \end{subfigure}
\caption{No noise experiment using QR factorization}
\label{fig:No noise}
\end{figure}


We next generate a noise field as another Gaussian random field, with
angular power spectrum $\eCov_{\ell} = \tau C_{\ell}, \; \ell=0,\ldots,L$ with
$\tau = 10^{-4}; 10^{-3}; 10^{-2}$. 
We then add the noise field to the original Gaussian random field,
mask it with the axially symmetric mask, then reconstruct the field
using one of the following approaches.

i) Using the QR factorization \eqref{eq:Q1R1}, we can simplify
\eqref{eq:alpha}, assuming $R_1$ is non-singular, by solving \eqref{eq:QReq} for $\balpha$.
The coefficient vector $\widehat{\ba}$ is then obtained
from equation \eqref{eq:afromalpha}. Note that in our experiments,
since $\eCov = \tau \aCov$, and both matrices $\eCov$ and $\aCov$
are diagonal, \eqref{eq:afromalpha} is just
\begin{equation}\label{eq:a al simp}
  \widehat{\ba} =  \frac{1}{1+\tau} \balpha.
\end{equation}
The reconstructed fields $\widehat{a}(\br)$ and error fields
$\widehat{a}(\br) - a(\br)$ using this approach are plotted
for $\br\in\Sph$ in
Figures~\ref{fig:Noise sig 1e-4 QR} and \ref{fig:Noise sig 1e-2 QR}
for the two noise levels $\tau = 10^{-4}$ and $\tau = 10^{-2}$. 

ii) In the second approach, we use the regularised equation \eqref{eq:alphareg} and \eqref{eq:afromalpha} with $\Gamma$ being the $(J+1)\times(J+1)$ identity matrix and $\Sigma = \nu I$ with $I$ being 
the $(L+1)\times (L+1)$ identity matrix. 
The optimal value of $\nu$ is found by minimizing the squared $\ell_2$ error
\begin{equation}\label{eq:l2err}
 \sum_{\ell = 0} ^ \infty \sum_{m= -\ell}^\ell (\widehat{a}_{\ell,m} - a_{\ell, m})^2. 
\end{equation}
By Parseval's theorem, minimising \eqref{eq:l2err} is equivalent to minimising
$\|\widehat{a} - a\|^2_{L_2}$.
The optimal value of $\nu$ is estimated via a grid search procedure over the values
$\nu_0 = 10^{-15}; 1; 10; 10^2; 10^3; 10^4; 10^5$, then, for $\nu_0 = 10^{-15}$, zooming in to
a finer grid $$\nu_k = \{ 10^{-15}+k\times 10^{-16}: k=0,5,10,15,20,25\}.$$
For our particular experiments,
we found the optimal value of $\nu$ to be 
$2\times 10^{-15}$. However, the reconstructed fields obtained with this method were found to be consistently less accurate than those from the QR factorisation approach.  For that reason the second approach was judged to be uncompetitive, and as a result the numerical results from this method are not reported.

iii) Inspired by a recent publication on a related problem 
\cite{LiChen2023}, we used the \texttt{SPGL1} package \cite{BergFriedlander:2008,spgl1site}
to solve the following optimisation problems:
\begin{equation}\label{eq-SPGL1}
\min_{ \ba \in \mathbb{C}^n} \|\ba\|_{\ell_2} \quad\text{ s.t. } \quad  
\| E \ba - \bb^v \|_{\ell_2} \le \rho,\quad \rho := \|\boldsymbol{\varepsilon}^v\|_{\ell_2},
\end{equation}
where $n=L+1-m$ and $m=0,1,\ldots,L$. Solving \eqref{eq-SPGL1} with the $\ell_2$ norm objective minimizes
the energy of the approximating field with the accuracy determined by the parameter $\rho$.
\texttt{SPGL1} can also solve \eqref{eq-SPGL1} with the objective $\| \ba \|_{\ell_1}$ to find a (group) sparse approximation $\ba$,
  still with an accuracy governed by the parameter $\rho$.
We remark that \eqref{eq-SPGL1} corresponds to (6.2) in \cite{LiChen2023}. We also tried $\rho =k\|\boldsymbol{\varepsilon}^v\|_{\ell_2}$ with $k=0.5; 0.25$ and observed a slight
decrease in the overall error but an increase in the error in the masked region. Note that we primarily focus on recovering the underlying field, even in masked regions where no data is directly available. On the other hand, the focus of Li and Chen \cite{LiChen2023} is really on finding~(group)~sparse~approximations.
\vspace{.4cm}

To obtain measures of the error using different approaches we define
\[
  \|a\|_{\ell_2} =  
  \left(
  \frac{4\pi}{N_{\rm pix}}
   \sum_{i=1}^{N_{\rm pix}} |a(\br_i)|^2 \right)^{1/2},
\]
while the root mean square error of a re-constructed map
is defined to be
\[
  {\rm RMSerr}(\widehat{a}) =  
  \left
  ( \frac{4\pi}{N_{\rm pix}}\sum_{i=1}^{N_{\rm pix}} |\widehat{a}(\boldsymbol{r}_i) - 
  a(\boldsymbol{r}_i)|^2 \right)^{1/2},
\]
where $N_{\rm pix}$ is the number of pixels of
the map in \texttt{HEALPix} format, and $\br_i$ is the $i$th Healpix point.
In our numerical experiments,
$N_{\rm pix}=50,331,648$ and $\|a\|_{\ell_2} = 50.265$.
The relative errors are defined to be 
\[
{\rm rel\;\; err}:= {\rm RMSerr}(\widehat{a})/\|a\|_{\ell_2}. 
\]
In Table~\ref{tab:l2_errors}, Columns $2$ and $3$ of show the root mean square errors and relative errors between the reconstructed maps and the original map using the QR factorisation approach while columns $4$ and $5$ show the root mean square errors and relative errors between the reconstructed maps and the original map using the \texttt{SPGL1} algorithm. It can be seen that the QR factorisation approach delivers markedly better results.

We also show the running time (in seconds) of the two algorithms, which are of the same order. The scripts were written in Matlab and Python and ran on a Linux desktop with an Intel Core i9-12900 processor and 32GB of RAM. The matrices $E^{(m)}$ were pre-computed using the high-performance computer cluster Katana~\cite{Katana} provided by UNSW, Sydney with the help of the \texttt{sympy} package \cite{sympy}. The source code of the numerical experiments is available on GitHub \footnote{\url{https://github.com/qlegia/RemovingMask}}. 
\begin{table}
    \centering
    \begin{tabular}{|c|c|c|c|c|c|c|}
    \hline
    $\eCov_{\ell}$ 
    & RMSerr$(\widehat{a})$ &   rel err & time
    & RMSerr$(\widehat{a})$ &   rel err& time \\
    \hline
    $0$ & $3.913$ & $0.078$ & $0.7$ &$20.815$ & $0.414$ & $0.5$\\
    \hline
$10^{-4} C_{\ell}$ & $3.949$ & $0.079$ & $0.7$ & $23.253$ & $0.463$& $0.5$\\
$10^{-3} C_{\ell}$ & $4.213$  & $0.084$ &$0.7$& $24.070$ & $0.479$& $0.5$\\
$10^{-2} C_{\ell}$ & $6.364$  & $0.127$ &$0.7$& $26.147$ & $0.520$&$0.5$\\
\hline
& \multicolumn{3}{|c|}{QR} & \multicolumn{3}{|c|}{SPGL1} \\
    \hline
    \end{tabular}
    \caption{RMS errors, relative errors and running times of reconstructed maps using QR factorization vs. SPGL1 algorithm. Running times were measured in seconds and averaged over $10$ runs.}    \label{tab:l2_errors}
\end{table}



We also determine  
contributions to the error from two key regions of the sphere:
\[
  {\mathcal R}_0 := \{ \br \in \Sph: v(\br) = 0\}, \qquad
  {\mathcal R}_1 := \{ \br \in \Sph: v(\br) > 0\}.
\]
In ${\mathcal R}_0$ the mask is zero,
so no information is available at these points,
while in ${\mathcal R}_1 $ the mask is non-zero, so at least some information is available at all the Healpix points in ${\mathcal R}_1$.
Let $N_j$ be the number of Healpix points in region ${\mathcal R}_j$  .
In our experiments, $N_0 = 8,740,864$ and $N_1 = 41,590,784$.

In the following, we let, for $j = 0, 1$,
\[
   \|a\|_{\ell_2({\mathcal R}_j)} =  
  \left(
  \frac{4\pi}{N_j}
   \sum_{\boldsymbol{r} \in {\mathcal R}_j} |a(\boldsymbol{r})|^2 \right)^{1/2},
\quad
{\rm RMSerr_j} (\widehat{a} ) =
\left( \frac{4\pi}{N_j}\sum_{\br  \in {\mathcal R}_j}
  |\widehat{a}(\boldsymbol{r}) - 
  a(\boldsymbol{r})|^2 \right)^{1/2}.
\]
The relative errors on each region ${\mathcal R}_j$ for $j=0,1$ are defined
by
\begin{equation}\label{eq:errRj}
{\rm rel\; err}_j = {\rm RMSerr}_j(\widehat{a})/\|a\|_{\ell_2({\mathcal R}_j)}.
\end{equation}
In our experiments, $N_0 = 8,740,864$, $N_1 = 41,590,784$,
$\|a\|_{\ell_2(\calR_0)}=51.071$ and $\|a\|_{\ell_2(\calR_1)}=50.094$.
\begin{table}
    \centering
    \begin{tabular}{|c|c|c|c|c|}
    \hline
    $\eCov_{\ell}$ 
    & RMSerr$_0(\widehat{a})$ & rel err$_0$ 
    & RMSerr$_1(\widehat{a})$ & rel err$_1$ \\
    \hline
    $0$ & $9.390$ & $0.184$ & $9.7 \cdot 10^{-5}$ & 
      $1.9 \cdot 10^{-6}$ \\
    \hline
    $10^{-4} C_{\ell}$ & $9.409$ &   $0.184$ & $0.519$ & $0.010$ \\
    $10^{-3} C_{\ell}$ & $9.470 $ &  $0.185$ & $1.622$  & $0.032$ \\
    $10^{-2} C_{\ell}$ & $10.457$ &  $0.205$  & $5.102$  & $0.102$ \\ 
    \hline
    \end{tabular}
    \caption{RMS errors of reconstructed maps 
    for different regions of the mask using QR factorisation, 
    where the relative error on each region is defined in \eqref{eq:errRj}. 
    }    \label{tab:l2_errors 0 and 1}
\end{table}


\begin{table}
    \centering
    \begin{tabular}{|c|c|c|c|c|}
    \hline
    $\eCov_{\ell}$ 
    & RMSerr$_0(\widehat{a})$ & rel err$_0$ 
    & RMSerr$_1(\widehat{a})$ & rel err$_1$ \\
    \hline
 $0$ & $49.781$ & $0.975$ & $1.867$  & $0.037$   \\
 \hline
$10^{-4} C_{\ell}$ & $51.383$ & $1.006$ & $9.973$ & $0.199$ \\
$10^{-3} C_{\ell}$ & $51.071$ & $1.001$ & $12.323$ & $0.246$ \\
$10^{-2} C_{\ell}$ & $51.071$ & $1.000$ & $16.721$ & $0.334$ \\
\hline
    \end{tabular}
    \caption{RMS errors of reconstructed maps 
    for different regions of the mask using 
    the \texttt{SPGL1} method. 
    } \label{tab:l2_errors 0 and 1 SPG}
\end{table}

It can be seen from Table~\ref{tab:l2_errors 0 and 1} that the QR reconstruction is of reasonable quality even in the masked region.
In contrast, Table~\ref{tab:l2_errors 0 and 1 SPG} shows
that the \texttt{SPGL1} algorithm has essentially no validity in the masked region, in that the root mean square relative error in region ${\mathcal R}_0$ is of the order of 100\%.
\section{Conclusion}

In this paper we have analysed a spectral method for recovering a scalar field from a masked and possibly noisy version of that field.  The quality of the recovery might be considered acceptable even in the presence of noise.  However, it is acknowledged that the quality will deteriorate as the noise level is increased and as the cutoff polynomial degree is increased from $100$.


\begin{figure}[ht]
\centering
\begin{subfigure}{\textwidth}
\hspace{2cm}
  \includegraphics[scale=0.55]{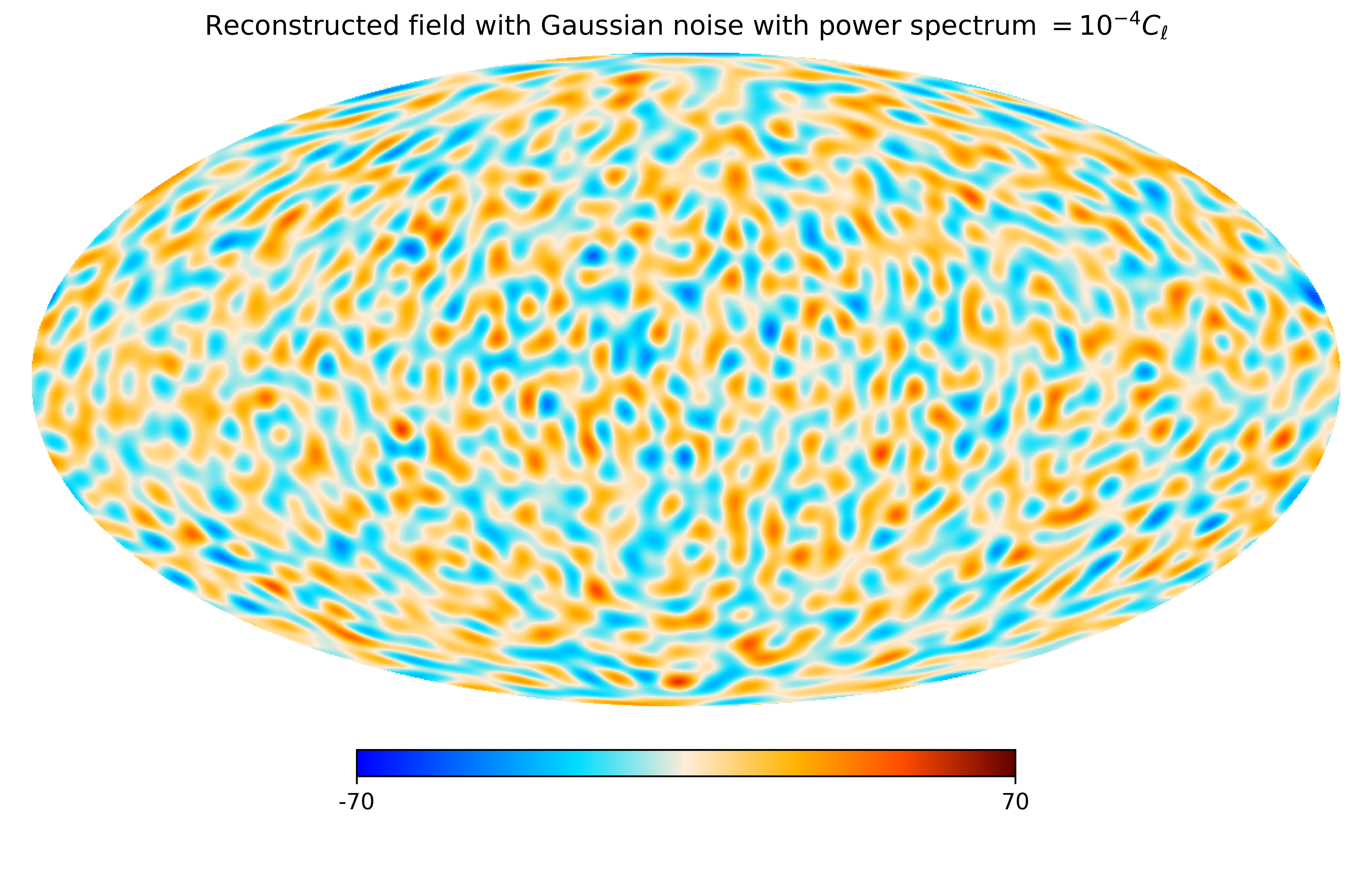}
   \caption{Reconstructed field}
   \label{fig:reconstructedRF2 Sig 1e-4 QR}
\end{subfigure}
\\
\begin{subfigure}{\textwidth}
\hspace{2cm}
   \includegraphics[scale=0.55]{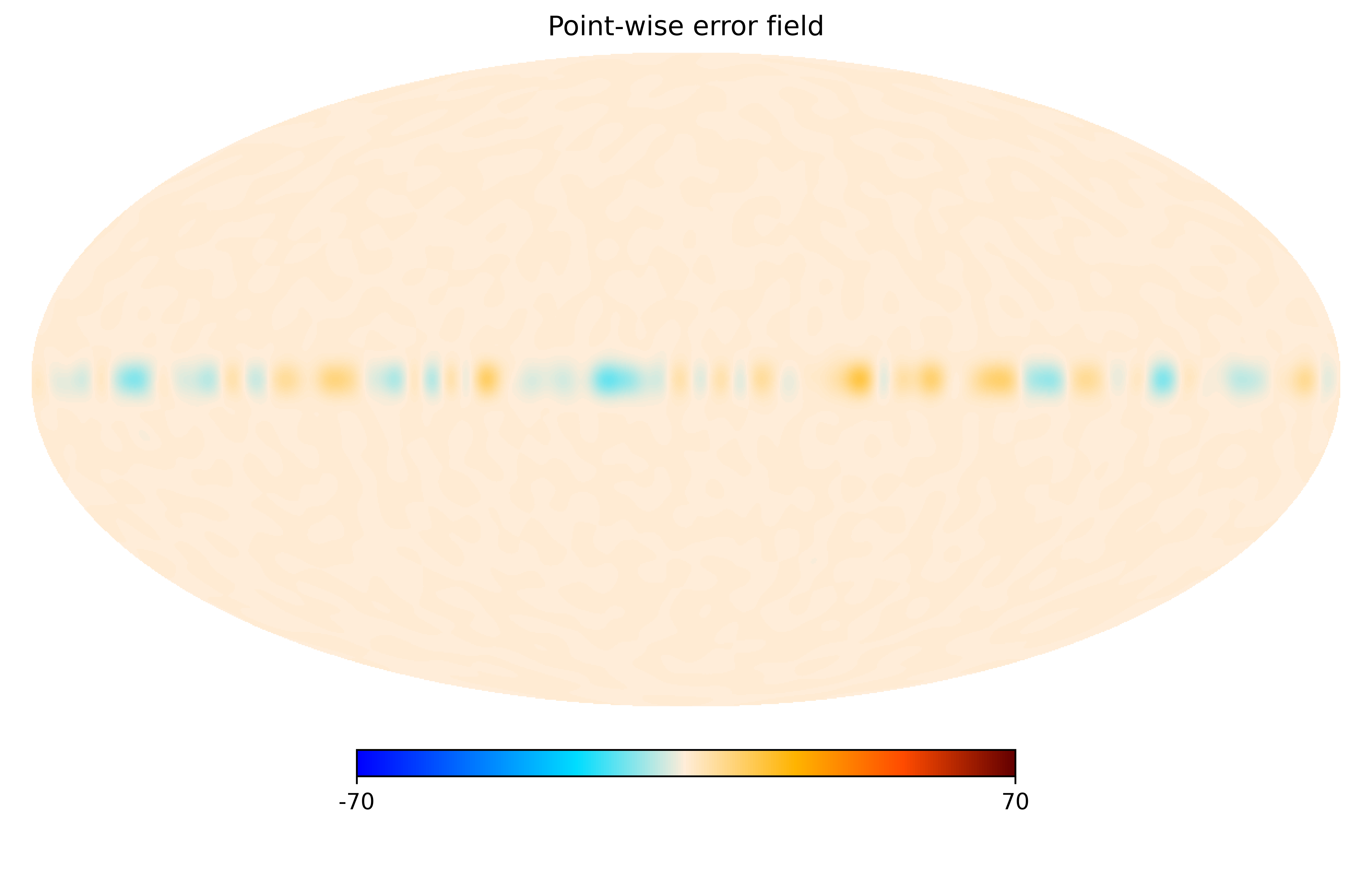}
   \caption{Error field}
   \label{fig:error2 Sig 1e-4 QR}
\end{subfigure}
\caption{QR method for Gaussian noise with angular power spectrum $\eCov_\ell = 10^{-4} C_\ell$}
\label{fig:Noise sig 1e-4 QR}
\end{figure}

\begin{figure}[ht]
\centering
\begin{subfigure}{\textwidth}
\hspace{2cm}
   \includegraphics[scale=0.55]{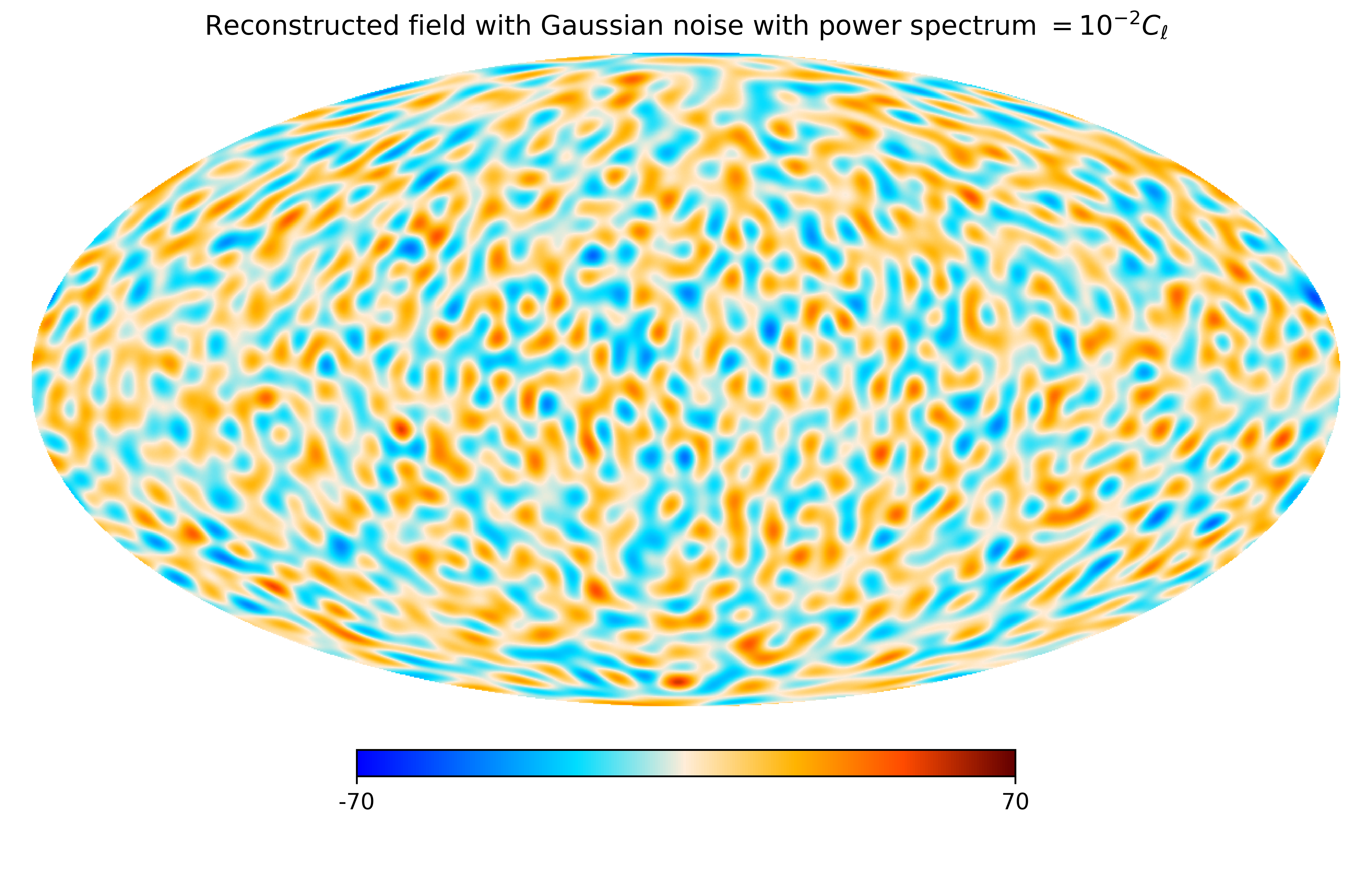}
   \caption{Reconstructed field}
   \label{fig:reconstructedRF2 Sig 1e-2 QR}
   \end{subfigure}
\\
\begin{subfigure}{\textwidth}
\hspace{2cm}
    \includegraphics[scale=0.55]{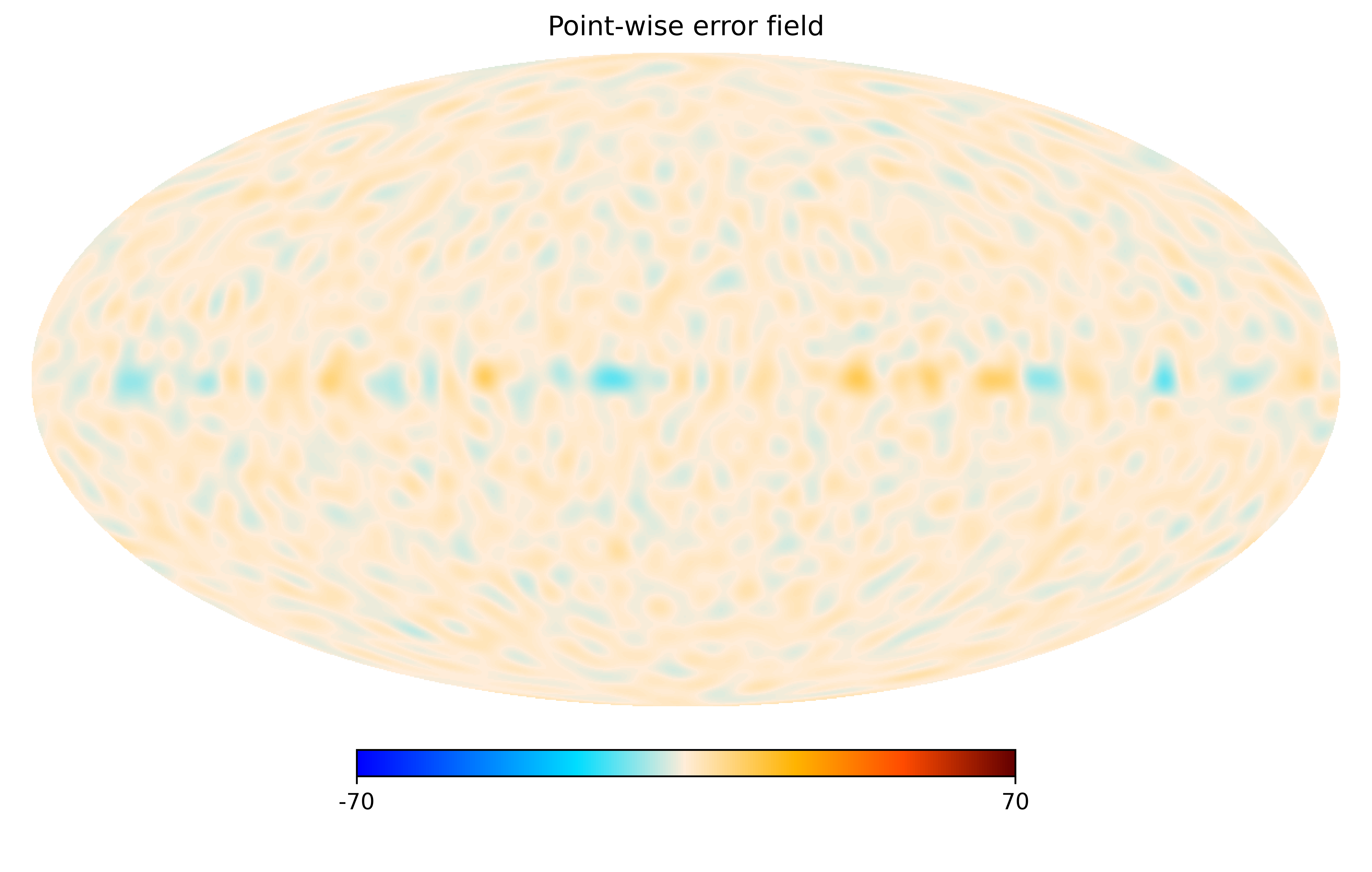}
    \caption{Error field}
   \label{fig:error2 Sig 1e-2 QR factorization}
   \end{subfigure}
\caption{QR method for Gaussian noise with angular power spectrum $\eCov_\ell = 10^{-2} C_\ell$}
\label{fig:Noise sig 1e-2 QR}
\end{figure}


\section*{Acknowledgements}
The assistance of Yu Guang Wang in the early stages of the project and constructive comments
from anonymous referees are gratefully acknowledged. 
This research includes computations using the computational cluster Katana supported by Research Technology Services at UNSW Sydney~\cite{Katana}. Some of the results in this paper have been derived using the \texttt{healpy} and \texttt{HEALPix} packages.

\bibliographystyle{plain}
\bibliography{reference}
\end{document}